\documentclass{article}
\usepackage{amssymb}
\usepackage{amsthm}
\usepackage{amsmath}
\usepackage{graphics}
\usepackage{xcolor}
\usepackage{csquotes}
\usepackage{graphicx}
\usepackage{csquotes}
\usepackage{float}
\usepackage{soul}
\usepackage[a4paper,
left=2.5cm, right=2cm,
top=3cm, bottom=3cm]{geometry}
\definecolor{blue}{rgb}{0.0, 0.313, 0.608}
\usepackage{hyperref}
\hypersetup{
    colorlinks,
    linkcolor={blue},
    citecolor={blue},
    urlcolor={orange}
}

\newtheorem{theorem}{Theorem}[section]
\newtheorem{lemma}[theorem]{Lemma}
\newtheorem{remark}[theorem]{Remark}
\newtheorem{proposition}[theorem]{Proposition}

\newtheorem{definition}[theorem]{Definition}

\newcommand\norm[1]{\left\lVert#1\right\rVert}
\newcommand{\R}{\mathbb{R}}
\newcommand{\N}{\mathbb{N}}	
\newcommand{\e}{\varepsilon}

\title{Traveling waves and transverse instability for the fractional Kadomtsev-Petviashvili equation}

\author{Handan Borluk\footnote{\texttt{handan.borluk@ozyegin.edu.tr}} \\  
{\small Ozyegin University, Department of Basic Sciences},
{\small  Cekmekoy, Istanbul,  Turkey} \\
\\
Gabriele Bruell\footnote{\texttt{gabriele.bruell@kit.edu, gabriele.brull@math.lth.se}} \\ 
 {\small Karlsruhe Institute of Technology, Department of Mathematics},
 {\small Karlsruhe, Germany}\\
  {\small Lund University, Centre for Mathematical Sciences, Lund, Sweden}\\
\\
Dag Nilsson\footnote{\texttt{nilsson@math.uni-sb.de}} \\ 
{\small Saarland University, Department of Mathematics},
{\small Saarbrücken, Germany}}

\begin{document}

\maketitle

\begin{abstract}

Of concern are traveling wave solutions for the fractional Kadomtsev--Petviashvili (fKP) equation. The existence of periodically modulated solitary wave solutions is proved by dimension breaking bifurcation. Moreover, the line solitary wave solutions and their transverse (in)stability are discussed.
Analogous to the classical Kadmomtsev--Petviashvili (KP) equation, the fKP equation comes in two versions: fKP-I and fKP-II. We show that the line solitary waves of fKP-I equation are transversely linearly instable. We also perform numerical experiments to observe the (in)stability dynamics of line solitary waves for both fKP-I and fKP-II equations.

\end{abstract}

\emph{Keywords:}  fractional Kadomtsev-Petviashvili equation, dimension breaking bifurcation, transverse instability, solitary waves, Petviashvili iteration, exponential time differencing. \\

\renewcommand{\theequation}{\arabic{section}.\arabic{equation}}
\setcounter{equation}{0}
\section{Introduction}
The present paper is devoted to the study of traveling waves and transverse (in)stability for the fractional Kadomtsev--Petviashvili (fKP) equation
\begin{equation*}\label{fkp-general}
	(u_t+uu_x-D^\alpha u_x)_x+\sigma u_{yy}=0,
\end{equation*}
where $\tfrac{1}{3}<\alpha \leq  2$ and $\sigma=\pm 1$.
Here the real function $u=u(t,x,y)$ depends on the spatial variables $x,y\in \R$ and the temporal variable $t\in \R_+$.
The linear operator $D^\alpha$ denotes the Riesz potential of order $\alpha$ in the $x$ direction and it is defined by multiplication with $|\cdot|^\alpha$ on the frequency space, that is
\begin{equation*}
	(D^\alpha f)~\widehat{ }~(t,\xi, \eta)=|\xi|^\alpha \hat{f}(t,\xi, \eta),
\end{equation*}
where the hat denotes the Fourier transform $\hat{f}(\xi)=\int_{\mathbb{R}}f(x)\mathrm{e}^{-\mathrm{i}x\xi}\ \mathrm{d}x$.
In case of $\alpha=2$  equation \eqref{fkp-general} takes the form of the classical Kadomtsev--Petviashvili (KP) equation
which was introduced by Kadomtsev and Petviashvili \cite{kp} as a weakly two-dimensional extension of the celebrated Korteweg--de Vries (KdV) equation,
\begin{equation*}\label{kdv}
	u_t+uu_x+u_{xxx}=0,
\end{equation*}
that is a spatially one-dimensional equation appearing in the context of small-amplitude shallow water-wave model equations. The KP equation is called KP-I when $\sigma =-1$ and KP-II when $\sigma=1$. Roughly speaking, the KP-I equation represents the case of strong surface tension, while the KP-II equation appears as a model equation for weak surface tension\footnote{Notice that by  the change of variables $u(x,y,t)\mapsto -u(x,y,-t)$ the parameter $\sigma$ shifts in front of $u_{xxx}$.}. Thus, the KP-I equation has limited relevance in the context of water waves, since strong surface tension effects are rather dominant in thin layers including viscous forces. However the KP-I equation also appears for instance as  a long wave limit for the Gross--Pitaevskii equation \cite{BGS}.
Analogously to the classical case, the  fKP  equation is a two-dimensional extension of the fractional Korteweg--de Vries (fKdV) equation
\begin{equation}\label{fkdv}
	u_t+uu_x-D^\alpha u_x=0
\end{equation}
and  \eqref{fkp-general} is referred to as the fKP-I equation when  $\sigma=-1$ and as the fKP-II equation when $\sigma=1$.
Notice that for
$\alpha =1$ in \eqref{fkp-general} we recover the KP-version of the Benjamin--Ono equation. During the last decade there has been a growing  interest in fractional regimes as the fKdV or the fKP equation \eqref{fkdv} (see for example \cite{albert1, duran, fonseca, franklenzmann, klein, linares1, linares2, natali, pava} and the references therein). Even though most of these equations are not derived by asymptotic expansions from governing equations  in fluid dynamics, they can be thought of as dispersive corrections.

\subsection{Some properties of the fKP equations}
Formally, the fKP equation does not only conserve the $L^2$--norm
\[
M(u)=\int_{\R^2} u^2\,d(x,y),
\]
but also the energy
\[
E_\alpha(u):=\int_{\R^2}\left( \frac{1}{2} (D^\frac{\alpha}{2} u)^2-\tfrac{1}{6}u^3-\tfrac{1}{2}\sigma(\partial_x^{-1}u_y)^2 \right)\, d(x,y).
\]
Here, the operator $\partial_x^{-1}$ is defined as a Fourier multiplier operator on the $x$-variable as
$
\widehat{\partial_x^{-1}u}(t,\xi,y)=\frac{1}{i\xi}\hat u(t,\xi,y).
$
Notice that the corresponding energy space
\[
X_\frac{\alpha}{2}(\R^2):=\{u\in L^2(\R^2)\mid D^\frac{\alpha}{2}u , \partial_x^{-1}u_y\in L^2(\R^2)\}
\]
includes a zero-mass constraint with respect to $x$.
We refer to \cite{linares1} for derivation issues and well-posedness results for the Cauchy problem associated with \eqref{fkp-general}.
In \cite{molinet3,molinet2} the authors established global well-posedness on the background of a non-localized solution (as for instance the line solitary wave solutions, which are localized in $x$-direction and trivially extended in $y$-direction) for the classical KP-I and KP-II equations. 
The fKP equation is invariant under the scaling
\[
u_\lambda(t,x,y)=\lambda^\alpha u(\lambda^{\alpha+1}t,\lambda x, \lambda^{\frac{\alpha+2}{2}}y),
\]
and $\|u_\lambda\|_{L^2}=\lambda^{\frac{3\alpha-4}{4}}\|u\|_{L^2}$. Thus, $\alpha=\frac{4}{3}$ is the $L^2$-critical exponent for the fKP equation. The ranges $\alpha>\frac{4}{3}$  and $\alpha<\frac{4}{3}$ are called \emph{sub}- and \emph{supercritical}, respectively.

\medskip

A traveling wave solution $u(t,x,y)=\phi(x-ct,y)$ of the of the fKP-I equation propagating in $x$-direction with wave speed $c>0$, satisfies  the steady equation
\begin{equation}\label{eq:steady_KPI}
	\left(-c\phi +\frac{1}{2}\phi^2-D^\alpha \phi\right)_{xx}-\phi_{yy}=0.
\end{equation}
A traveling wave solution $u(t,x)=Q_c(x-ct)$ of the one-dimensional fKdV equation such that  $Q_c(x-ct)\rightarrow 0$ as $|x-ct|\rightarrow \infty$, is called a \emph{solitary wave solution}. It has been shown that the fKdV equation possesses solitary wave solutions for $\alpha>\frac{1}{3}$ \cite{arnesen,weinstein}. It is clear that $Q_c$ is a $y$-independent solution of the steady KP-I equation \eqref{eq:steady_KPI} and we refer to such solutions as \emph{line solitary wave solutions} of the fKP equation. Solutions of the fKP-I equation with a traveling solitary wave profile in the $x$-direction and periodic in the $y$-direction are called \emph{periodically modulated solitary wave solutions}.

\subsection{Aim of the paper and results}
Our study on the fKP equation \eqref{fkp-general} is twofold. On the one hand we establish the existence of two-dimensional traveling waves for the fKP-I equation, which are periodically modulated solitary wave solutions.
On the other hand, we investigate the transverse instability of the line solitary wave solution for the fKP-I equation analytically and support this result by numerical experiments.
Both results have in common that they substantially rely on the spectral properties of the appearing linearized operator
\begin{equation}\label{eq:LL}
	L:=-\partial_x \left( D^\alpha -Q_c+c \right) \partial_x.
\end{equation}
We show that on appropriate function spaces the operator $L$ is self-adjoint with a sole simple negative eigenvalue and the rest of the spectrum is included in $[0,\infty)$.
Eventually, we also investigate  the transverse stability for the fKP-II equation numerically.
	Here, we give a brief summary of the results and state the main theorems.

\paragraph*{Existence of periodically modulated solitary wave solutions}
We prove that periodically modulated line solitary wave solutions for the fKP-I equation emerge from the line solitary wave solutions in a \emph{dimension breaking} bifurcation, in which a spatially inhomogeneous solution emerges from a solution which is homogeneous in at least one spatial variable, see \cite{haragus-kirchgassner}. Our result reads:

\begin{theorem}[Existence of periodically modulated solitary waves]\label{thm:main1}
	Let $\tfrac{1}{3}<\alpha< 2$ and $Q_c(x-ct)$ be a ground state solitary wave solution of the fKdV equation with wave speed $c>0$. Then there exists an open neighborhood $I\subset \R$ of the origin and a continuously differentiable branch
		\[
	\{ (\phi(s),\omega(s))\mid s\in I\}\subset C_{\text{b}}(\R; H^{1+\alpha}(\R))\times (0,\infty)
	\]

	of solutions of the steady KP-I equation \eqref{eq:steady_KPI}, which are even in $x$ and $2\pi/\omega(s)$-periodic in $y$, with $\omega(0)=\sqrt{|\lambda|}$, where $\lambda<0$ is the sole simple negative eigenvalue of the linear operator L.
\end{theorem}

In order to prove the above theorem we first reformulate the fKP-I equation as a dynamical system
\begin{equation}\label{dynsysintro}
	W_y=\mathcal{L}W+\mathcal{N}(W),
\end{equation}
where the periodic variable $y$ plays the role of time (see \eqref{dynsys1}.).
If $\mathcal{L}$ has non-resonant eigenvalues $\pm\mathrm{i}\omega_0$, the existence of periodic solutions of \eqref{dynsysintro} with frequency close to $\omega_0$ can be shown by using the classical Lyapunov-center theorem. However, in the situation at hand the essential spectrum of $\mathcal{L}$ includes zero, 
which violates the non-resonance condition. In \cite{iooss} it was observed that in the proof of the Lyapunov-center theorem the condition that $\mathcal{L}$ is invertible can be relaxed to require $\mathcal{L}$ to be invertible only on the range of $\mathcal{N}$. This leads to the so called Lyapunov-Iooss theorem (see Theorem \ref{lyap}), which we apply to prove the Theorem \ref{thm:main1}.

\medspace

\emph{Related reults:}
There are several results on existence of periodically modulated solitary waves for the water-wave problem and related model equations. Tajiri and Murakami \cite{tm} found explicit periodically modulated solitary wave solutions of the KP-I equation. A family of solutions for a generalized KP-I equation were found by Haragus and Pego in \cite{haraguspego}. Groves, Haragus, and Sun proved in \cite{ghs} existence of periodically modulated gravity-capillary solitary waves emerging from a KdV-type solitary wave in a dimension breaking bifurcation, for strong surface tension. Later on, Groves, Sun, and Wahlén proved in \cite{gsw1} the existence of periodically modulated gravity-capillary solitary waves  emerging from a NLS-type solitary wave,
for weak surface tension. In addition they showed in \cite{gsw1} that the Davey--Stewartson equation possesses the same type of solutions. In \cite{mw} Milewski and Wang computed numerically periodically modulated gravity-capillary solitary waves in infinite depth, with an NLS-type solitary wave profile.

\paragraph*{Transverse (in)stability of line solitary wave solutions}
Transverse stability of a two-dimensional traveling wave means stability with respect to perturbations which depend not only on the direction of propagation, but also on its transverse direction. This is unlike the one-dimensional stability, which relates to perturbations depending solely on the direction of propagation.
We investigate the transverse instability of line solitary wave solutions for the fKP-I equation analytically and support those results by numerical experiments.
 For the definition of linear transverse instability we refer to the discussion in Section \ref{S:2} and Definition \ref{def}.
 Our result reads:

\begin{theorem}[Linear transverse instability for the fKP-I equation]\label{thm:main2}
	Let $\frac{1}{3}<\alpha< 2$ and $u(t,x)=Q_c(x-ct)$ be a ground state solitary solution of the fKdV equation with wave speed $c>0$.
	Then $Q_c$ is linearly unstable under the evolution of the fKP-I equation with respect to transverse perturbations, which are localized in $x$-direction and bounded in $y$-direction.
	
\end{theorem}

The proof is based on criteria formulated by Rousset and Tzvetkov in \cite{rousset}. We would like to remark that the instability is due to sufficiently large transverse frequencies. In fact the frequency must be large enough to guarantee that the kernel of a respective linear operator is simple. This operator is a bounded perturbation of $L$ in \eqref{eq:LL} shifting the sole negative eigenvalue of $L$ to the origin.
 In addition to the analytical result, we provide a numerical analysis for the instability of line solitary waves for the fKP-I equation, as well as numerical evidence for stability under the fKP-II flow.
To this end, we first generate the line solitary wave solutions numerically by Petviashvili iteration method. Then we use a Fourier spectral method combined with an exponential time differencing scheme to investigate the time evolution of the line solitary wave solutions with respect to a given perturbation.

\medspace

\emph{Related reults:}
In 1970,  Kadomtsev  and Petviashvili introduced the classical KP equations
\begin{equation*}\label{kp}
	(u_t+uu_x+u_{xxx})_x+\sigma u_{yy}=0,
\end{equation*}
to  investigate the stability of solitary wave solutions of the KdV equation with respect to transverse effects.
The two versions of the KP equation, KP-I and KP-II, admit quite different transverse dynamics.
The earliest result goes back to Zakharov \cite{zakharov} in 1973, were it is proved that the solitary wave solution of the KdV equation is nonlinearly unstable under the evolution of the KP-I flow. An alternative proof is given by Rousset and Tzvetkov \cite{rousset2} in 2009, and linear instability has been shown by Alexander, Pego, and Sachs \cite{alexander} in 1997. The stability dynamics for the KP-II flow appear to be different. It was shown by Mizumachi and Tzvetkov \cite{mizumachi2} and by Mizumachi \cite{mizumachi1} that the line solitary wave solutions are stable under periodic as well as localized transverse perturbations.

In the context of periodic line traveling waves, a spectral instability criterion for a generalized KP equation is established in \cite{johnson}.  Furthermore, in \cite{haragus2} spectral (in)stability of small amplitude periodic line traveling waves for the KP equation is investigated, obtaining a transverse instability result for KP-I and a stability result for the KP-II equation. In \cite{haraguspelinovsky} a general \emph{counting unstable eigenvalues} result is presentet for Hamiltonian systems. As an application, the authors of \cite{haraguspelinovsky} prove that periodic line traveling waves are transversly spectrally stable under the KP-II flow with respect to general bounded perturbations and that they are transversly linearly stable with respect to double periodic perturbations. Transverse instability of periodic and generalized line solitary wave solutions is shown in \cite{haragus} in the context of a fifth-order KP model. The term generalized solitary wave solution refers to solutions, which decay exponentially to periodic waves at infinity.

We also would like to mention that the solitary wave solutions for the fKdV equation, which are known to exist for $\alpha>\frac{1}{3}$, are orbitally stable as a one-dimensional solution when $\alpha>\frac{1}{2}$ and spectrally unstable for $\frac{1}{3}<\alpha<\frac{1}{2}$ \cite{albert1,linares2, pava} . The stability of solitary waves is an open problem for the critical case $\alpha=\frac{1}{2}$.

\medskip

\subsection{Organization of the paper}
Section \ref{S:1} is devoted to the existence of traveling wave solutions for the fKP-I equation. We first review some results on the line solitary wave solutions of the fKP equations. Then, by applying a generalized Lyapunov center theorem we prove Theorem \ref{thm:main1} on the
periodically modulated solitary wave solutions.
In Section \ref{S:2} we show that the line solitary wave solutions of the fKP-I equation are linearly unstable with respect to localized transverse perturbations.  Eventually, a numerical study on the (in)stability dynamics of  the line solitary wave solutions
under KP-I as well as the KP-II flows is presented in Section \ref{S:3}.

\bigskip

\section{Traveling waves for the fractional KP-I equation}\label{S:1}

We prove the existence of periodically modulated soliary solution for the fKP-I equation. The proof is based on methods from dynamical systems and often refered to \emph{dimension breaking} bifrucation.

\subsection{Line solitary wave solutions}

The line traveling wave solutions of the fKP equations are the solitary wave solutions of the fKdV equation. If $u(x,t)=Q_c(x-ct)$ is a solitary wave solution  of the fKdV equation with wave speed $c>0$, then $Q_c$
satisfies  the ordinary differential equation
\begin{equation}\label{fkpODE}
	D^{\alpha} Q_c+cQ_c-\tfrac{1}{2}Q_c^{2}=0
\end{equation}
after integration. Notice that the integration constant, say $B\in \R$, can always set to be zero due the Galilean invariance of the fKdV equation
	\[
	\phi \mapsto \phi + \gamma,\qquad c \mapsto c + \gamma,\qquad B\mapsto B-\gamma\left(c+\tfrac{1}{2}\gamma\right).
	\]
By the scaling $Q_c(z)=2cQ(c^{1/\alpha} z)$ equation \eqref{fkpODE} becomes
\begin{equation}\label{ODE}
	D^{\alpha} Q+Q-Q^2=0.
\end{equation}

\begin{proposition}[Existence of solitary waves for fKdV]\label{existence-kdv}
	If $\alpha>\frac{1}{3}$, then there exists a solution $Q\in H^{\frac{\alpha}{2}}(\R)$ of equation \eqref{ODE}.
\end{proposition}

Weinstein \cite{weinstein} proved the above statement for $\alpha \geq 1$. Recently, Arnesen \cite{arnesen} extended the existence result for solitary waves in a much more general framework including the fKdV equation for $\alpha>\frac{1}{3}$.

If $Q$ is a nontrivial solution of \eqref{ODE}, then it is a critical point of the Weinstein functional
\begin{equation*}
	J^{\alpha}(u)=\left(\int_{\mathbb{R}} |u|^{3} dx  \right)^{-1}
	\left(\int_{\mathbb{R}} |D^{\alpha/2}u|^{2}  dx  \right)^{\frac{1}{2\alpha}}
	\left(\int_{\mathbb{R}}  |u|^{2} dx  \right)^{\frac{\alpha-1}{2\alpha+1}}.
\end{equation*}
A nontrivial critical point $Q$ of the Weinstein functional $J^\alpha$ that is in addition even, positive and solves the minimization problem
\begin{equation*}
	J^{\alpha}(Q)=\inf \{J^{\alpha}(u)|u\in H^{\alpha/2}(\mathbb{R}) \backslash\{0\}\}
\end{equation*}
is called a \emph{ground state} for \eqref{ODE}.
The following result is from Frank and Lenzmann \cite{franklenzmann}:

\begin{theorem}[Uniqueness and structural  properties of ground states for the fKdV equation]\label{prop_Q}
	Let $\alpha\in (\tfrac{1}{3},2)$. Then there exists a unique ground state solution $Q\in H^{\frac{\alpha}{2}}(\R)$  of \eqref{ODE}. Moreover, the following holds true:
	\begin{itemize}
		\item[i)] $Q$ belongs to the class $H^{1+\alpha}(\R)\cap C^{\infty}(\R)$. Furthermore, it is symmetric and has exactly one crest.
		\item[ii)] $Q$ is and strictly decreasing from its crest and it satisfies the following decay estimate:
		\[
		|Q(x)| \lesssim \frac{1}{1+|x|^{\alpha+1}}.
		\]
	\end{itemize}
\end{theorem}

In view of Theorem \ref{prop_Q} the scaling $Q_c(z)=2cQ(c^{1/\alpha} z)$ yields that the equation \eqref{fkpODE} also has a unique  ground state solution $Q_c$ that is smooth, symmetric and has exactly one crest.

\medskip

\subsection{Existence of periodically modulated solitary wave solutions}

In this subsection we prove Theorem \ref{thm:main1} on the existence of partially localized two-dimensional traveling waves, which are periodically modulated solitary waves.
We consider the fKP-I equation
\begin{equation*}\label{fkp1}
	(u_t+uu_x-D^\alpha u_x)_x-u_{yy}=0.
\end{equation*}
The travelling wave ansatz $u(t,x,y)=\phi(x-ct,y)$ yields the following steady version of the fKP-I equation
\begin{equation}\label{tfkp1}
	\left(-c\phi+\tfrac{1}{2}\phi^2-D^\alpha \phi\right)_{xx}-\phi_{yy}=0.
\end{equation}
According to Theorem \ref{prop_Q}, for  $\tfrac{1}{3}<\alpha<2$, equation \eqref{tfkp1} possesses a unique line solitary wave solution $Q_c$, where $Q_c$ is a solitary wave solution of the fKdV equation.
We are seeking solutions of \eqref{tfkp1} of the form $\phi(x,y)=Q_c(x)+v(x,y)$, where $v$ is periodic in $y$ and localized in $x$. Inserting this ansatz in \eqref{tfkp1} gives the following equation for $v$:
\begin{equation}\label{v_eq_aux}
	\left(-cv+\tfrac{1}{2}(v^2+2vQ_c)-D^\alpha v\right)_{xx}-v_{yy}=0.
\end{equation}
If we let $v=w_x$ and integrate with respect to $x$, equation \eqref{v_eq_aux} becomes
\begin{equation*}
	\left(-cw_x+\frac{1}{2}w_x(w_x+2Q_c)-D^\alpha w_x\right)_x-w_{yy}=0,
\end{equation*}
which can be written as
\begin{equation}\label{v_eq_main}
	Lw+\tfrac{1}{2}\left(w_x^2\right)_x-w_{yy}=0,
\end{equation}
where
\begin{equation}\label{eq:L}
	L:=-\partial_xM_c\partial_x, \qquad  	M_c:=D^\alpha-Q_c+c.
\end{equation}
Note that $M_c$ is the linearization of \eqref{fkpODE} around $Q_c$.
In order to prove the existence of periodically modulated solitary wave solutions for the fKP-I equation, we view the stationary problem \eqref{v_eq_main} as an evolution equation, where the periodic spatial direction $y$ takes the role of the time variable.
We rewrite equation \eqref{v_eq_main} as a dynamical system of the form
\begin{equation}\label{dynsys1}
	W_y=\mathcal{L}W+\mathcal{N}(W),
\end{equation}
where $W=(W_1,W_2):=(w_y,w)$,
\begin{equation*}
	\mathcal{L}=\left(\begin{array}{cc}
		0 & L\\
		1 & 0
	\end{array}\right),\qquad  \mathcal{N}(W)=\left(\begin{array}{c}
		\frac{1}{2}(({W_{2}}_x)^2)_x\\
		0
	\end{array}\right).
\end{equation*}

Let $X=L_{{odd}}^2(\mathbb{R})\times H_{odd}^{1+\frac{\alpha}{2}}(\mathbb{R})$, $Z=H_{odd}^{1+\frac{\alpha}{2}}(\mathbb{R})\times H_{{odd}}^{2+\alpha}(\mathbb{R})$, where the subscript odd denotes restriction to odd functions in the respective function spaces, with norms
\begin{align*}
	\norm{W}_X=\norm{W_1}_{L^2}+\norm{W_2}_{H^{1+\frac{\alpha}{2}}},\qquad
	\norm{W}_Z=\norm{W_1}_{H^{1+\frac{\alpha}{2}}}+\norm{W_2}_{H^{2+\alpha}}.
\end{align*}
Then $\mathcal{L}$ is an unbounded, densely defined and closed operator on $X$ with domain $Z\subset X$. In order to prove the existence of small-amplitude periodic solutions of \eqref{dynsys1} we use the following generalization of the reversible Lyapunov centre theorem \cite[Theorem 1.1]{bagri}, which we refer to as the Lyapunov--Iooss theorem.

\begin{theorem}[Lyapunov-Iooss]\label{lyap}
	Consider the differential equation
	\begin{equation}\label{lyapunov-iooss-eq}
		\dot{W}=\mathcal{L}W+\mathcal{N}(W),
	\end{equation}
	in which $W$ belongs to a Banach space $X$ and suppose that
	\begin{itemize}
		\item[(i)] $\mathcal{L}\colon {D}(\mathcal{L})\subset {X}\rightarrow {X}$ is a densely defined, closed linear operator.
		\item[(ii)] There is an open neighborhood ${U}$ of the origin in ${D}(\mathcal{L})$ (regarded as a Banach space equipped with the graph norm) such that $\mathcal{N}\in C_{b,u}^3({U},{X})$ and $\mathcal{N}(0)=0,\ \mathrm{d}\mathcal{N}[0]=0$.
		\item[(iii)] Equation \eqref{lyapunov-iooss-eq} is reversible: there exists a bounded operator $\mathcal{S}$ on $X$ such that $\mathcal{S}\mathcal{L}W=-\mathcal{L}\mathcal{S}W$ and $\mathcal{S}\mathcal{N}(W)=-\mathcal{N}(\mathcal{S}W)$ for all $W\in {U}$.
		\item[(iv)] For each $W^*\in {U}$ the equation
		\begin{equation*}
			\mathcal{L}W=-\mathcal{N}(W^*)
		\end{equation*}
		has a unique solution $W\in {D}(\mathcal{L})$ and the mapping $W^*\mapsto W$ belongs to $C_{b,u}^3({U},{D}(\mathcal{L}))$.
	
	\end{itemize}
Suppose further that there exists $\omega_0\in \R$ such that
	\begin{itemize}
		\item[(v)]$\pm \mathrm{i}\omega_0$ are nonzero simple eigenvalues of $\mathcal{L}$.
		\item[(vi)] $\mathrm{i} n \omega_0\in \rho(\mathcal{L})$ for $n\in \mathbb{Z}\backslash \{-1,0,1\}$.
		\item[(vii)] $\norm{(\mathcal{L}-\mathrm{i}n\omega_0I)^{-1}}_{{X}\rightarrow {X}}=o(1)$ and $\norm{(\mathcal{L}-\mathrm{i}n\omega_0I)^{-1}}_{{X}\rightarrow {D}(\mathcal{L})}=O(1)$ as $|n|\rightarrow \infty$.
	\end{itemize}
	Under these hypotheses there exists an open neighborhood $I$ of the origin in $\mathbb{R}$ and a continuously differentiable branch $\{(W(s),\omega(s))\}_{s\in I}$ in $C_{\text{b}}^1(\mathbb{R},X)\cap C_{\text{b}}(\mathbb{R},{D}(\mathcal{L}))$ of solutions to \eqref{lyapunov-iooss-eq}, which are $\frac{2\pi}{\omega(s)}$-periodic, with $\omega(0)=\omega_0$.
\end{theorem}

\subsection{Proof of Theorem \ref{thm:main1}}
Let us fix $\alpha\in (\frac{1}{3},2)$.
In order to prove Theorem \ref{thm:main1}, we verify the hypotheses of Theorem \ref{lyap} in the context of the dynamical system formulation for the fKP-I equation \eqref{dynsys1},
 starting with investigating the spectrum of the operator $L$ defined in \eqref{eq:L}.

\begin{lemma}\label{lem:spectrum_L}
	The operator $L\colon H^{2+\alpha}(\mathbb{R})\rightarrow L^2(\mathbb{R})$ is selfadjoint. Moreover, the spectrum of $L$ consists of one simple negative eigenvalue $\lambda$ and the essential spectrum $\sigma_{ess}(L)=[0,\infty)$.
	\end{lemma}

\begin{proof}
	Let us start by showing that $L$ is a selfadjoint operator.  Since $-D^{2+\alpha}:H^{2+\alpha}(\R)\to L^2(\R)$ is a symmetric operator with $\pm \mathrm{i} \in \rho(-D^{2+\alpha})$, we infer that $-D^{2+\alpha}$ is self-adjoint. We use a perturbation argument to show that the operator $L=-D^{2+\alpha}+P$ with  \[P:=-\partial_x(c-Q_c)\partial_x\]
	is also self-adjoint.
	Due to the Rellich--Kato theorem \cite[Theorem 5.28]{weidmann}, the operator $L$ is self-adjoint if
	$P$ is symmetric and \emph{$D^{2+\alpha}$}-bounded with $D^{2+\alpha}$-bound less than one, that is if there exist constants $a\in [0,1)$ and $b\geq 0$, such that
	\[
	\|Pw\|_{L^2}\leq a\|D^{2+\alpha}w\|_{L^2}+b\|w\|_{L^2}\qquad \mbox{for any}\qquad w\in H^{2+\alpha}(\R).
	\]
	It is clear that $P$ is a symmetric operator.
	Moreover, if $w\in H^{2+\alpha}(\R)$, then
	\begin{align*}
		\|Pw\|_{L^2}&\leq \|(c-Q_c)\partial_x^2w\|_{L^2}+\|Q_c^\prime \partial_x w\|_{L^2}\\
		&\leq \|c-Q_c\|_\infty\|\partial_x^2 w\|_{L^2}+\|Q_c^\prime\|_{H^\alpha}\|\partial_x w\| _{L^2}\\
		&\leq \e\|D^{2+\alpha}w\|_{L^2}+c(\e)\|w\|_{L^2},
	\end{align*}
for some $\e>0$ and $c(\e)>0$.
	Here we use that $Q_c\in H^{1+\alpha}(\R)$ due to Theorem \ref{prop_Q}, and Lemma \ref{lemmaA}. The constant $\e>0$ can be chosen to be less than one on the cost of $c(\e)>0$. Then it follows that $P$ is $D^{2+\alpha}$-bounded with $D^{2+\alpha}$-bound $\e<1$ and
	$L:H^{2+\alpha}(\R)\to L^2(\R)$ is self-adjoint.
	 Now, we  show that $L$ has a simple negative eigenvalue and the rest of the spectrum is equal to $[0,\infty)$. It is clear that $\sigma_{ess}(L)=[0,\infty)$. Hence, we are only left to prove that $L$ has exactly one simple negative eigenvalue.
	In \cite{franklenzmann} the authors prove that $M_c$ has exactly one simple negative eigenvalue and the rest of the spectrum is included in $[0,\infty)$. Let $\psi_0$ denote the eigenfunction of $M_c$ corresponding to the negative eigenvalue $\lambda_0<0$. Then,
	\begin{equation}\label{eq_M}
		\langle M_cw,w\rangle \geq 0 \qquad \mbox{for all}\quad w\in (\psi_0)^\perp.
	\end{equation}
	We need to show that $L=-\partial_xM_c\partial_x$ has exactly one simple negative eigenvalue. First we prove that $L$ has at least one negative eigenvalue by using the min-max principle. Let $(w_n)_{n\in \N}\in H^{\alpha+2}(\R)$ be a sequence such that $\partial_x w_n \grave{\rightarrow}\psi_0 \in H^{\alpha+1}(\R)$. Then,
	\[
	\langle Lw_n,w_n\rangle = \langle M_c\partial_x w_n,\partial_x w_n\rangle \rightarrow \langle M_c \psi_0,\psi_0\rangle =\lambda_0\|\psi_0\|^2<0
	\]
	for $n\to \infty$.
	Denoting  the smallest eigenvalue of $L$ by $\lambda_m\in \R$, we deduce that
	\[
	\lambda_m=\inf_{\|u\|_{H^{\alpha+2}}=1}\langle Lw,w\rangle \leq \langle Lw_n,w_n\rangle<0
	\]
	for $n\in \N$ large enough. Thus, $L$ has at least one negative eigenvalue. By \eqref{eq_M}, we obtain that
	\[
	\langle Lw,w\rangle = \langle M_c w_x,w_x\rangle \geq 0
	\]
	for every $w\in H^{\alpha+2}(\R)$ such that $w_x \in  (\psi_0)^\perp$. This shows that $L$ is nonnegative on a codimension one subspace. Therefore, $L$ has at most one negative simple eigenvalue. Combining these results  we conclude that $L$ has exactly one negative simple eigenvalue.
	\end{proof}

\begin{remark}
	\emph{Our analysis is restricted to the case $\alpha\in(\tfrac{1}{3},2)$ and the solitary solution $Q_c$ of the fKdV equation being a ground state. The lower bound for $\alpha$ is simply due to the nonexistence result of solitary solutions in the energy space for fKdV when $\alpha\leq \tfrac{1}{3}$, cf \cite[Theorem 4.1]{lineares3}. The reason for the upper bound $\alpha<2$ and for considering ground state solutions $Q_c$ is because we rely strongly on the characterization of the spectrum of the operator $M_c$ defined in \eqref{eq:L}, which is studied in \cite{franklenzmann} for exactly this case. The case $\alpha=2$ corresponds to the classical KP equation and is already well-studied. If $\alpha> 2$, then there exist indeed solutions $Q_c$ of fKdV (cf. \cite{arnesen,weinstein}), however it is not clear what the spectrum of the corresponding  linearized operator $M_c$ around $Q_c$ is.}
\end{remark}

We  use the spectral information about $L$ in Lemma \ref{lem:spectrum_L} to determine the spectrum of $\mathcal{L}$.
\begin{proposition}\label{spectrum_of_L}
	Let $\lambda$ be the sole negative eigenvalue of $L$ found in Lemma \ref{lem:spectrum_L}. Then,
	the operator $\mathcal{L}:Z \to X$ has two simple eigenvalues $\pm \mathrm{i}\sqrt{|\lambda|}$ and the rest of the spectrum is included in $\mathbb{R}$.
\end{proposition}
\begin{proof}
	Let $\mu \in \mathbb{C}$.
	The resolvent equation $(\mathcal{L}-\mu I)W=W^*$ is equivalent to
	\begin{align*}
		\begin{cases}
			LW_2-\mu W_1&=W_1^*,\\
			W_1-\mu W_2&=W_2^*,
		\end{cases}\qquad \Longleftrightarrow \qquad
		\begin{cases}
			(L-\mu^2)W_2&=W_1^*+\mu W_2^*,\\
			W_1&=\mu W_2+W_2^*,
		\end{cases}
	\end{align*}
and so $\mu\in\rho(\mathcal{L})$ if and only if $\mu^2\in \rho(L)$ which in turn is equivalent with $\mu\in \mathbb{C}\backslash \{\mathbb{R}\cup \{\pm \mathrm{i}\sqrt{\lambda}\}\}$, by Lemma \ref{lem:spectrum_L}. In particular, if we put $W^*=0$, we find that $\mu$ is an eigenvalue of $\mathcal{L}$ if and only if $\mu^2$ is an eigenvalue of $L$, which implies that $\mathcal{L}$ has two simple eigenvalues $\pm \mathrm{i}\sqrt{|\lambda|}$.
\end{proof}

In view of Proposition \ref{spectrum_of_L} we conclude that $\mathcal{L}$ has two simple nonzero eigenvalues and $\mathrm{i}n\sqrt{|\lambda|}\in \rho(\mathcal{L})$ for $|n|>1$. Hence condition (v) and (vi) of the Lyaponov--Iooss theorem are satisfied. Next, we show that the resolvent estimates claimed in (vii) of Theorem \ref{lyap} hold true for the operator $\mathcal{L}$.
\begin{proposition}\label{resolvent-est} The operator $\mathcal{L}$ satisfies
	\begin{align*}
		\norm{(\mathcal{L}-\mathrm{i}n\sqrt{|\lambda|}I)^{-1}}_{X\rightarrow X}\lesssim \frac{1}{|n|},\qquad \mbox{and}\qquad
		\norm{(\mathcal{L}-\mathrm{i}n\sqrt{|\lambda|}I)^{-1}}_{X\rightarrow Z}\lesssim 1.
	\end{align*}
for $|n|>1$.
\end{proposition}
\begin{proof}
Here we use the same strategy as in the proof of Lemma 4.2 in \cite{gsw2}. Let us write
\begin{equation*}
L=-\partial_x M_c\partial_x=(-\mathrm{D}^\alpha\partial_x^2+1)+(-\partial_x(c-Q_c)\partial_x-1)=: L_1 + L_2,
\end{equation*}
where $L_1:=-\mathrm{D}^\alpha\partial_x^2+1$ and $L_2:=-\partial_x(c-Q_c)\partial_x-1$.  Accordingly, we split the matrix operator $\mathcal{L}$ into
\begin{equation*}
\mathcal{L}=\left(\begin{array}{cc}
0 & L_1\\
1 & 0
\end{array}\right)
+\left(\begin{array}{cc}
0 & L_2\\
0 & 0
\end{array}\right)=:\mathcal{L}_1 + \mathcal{L}_2.
\end{equation*}
On $X$ we define the inner product
\begin{equation*}
\langle (W_1,W_2),(U_1,U_2)\rangle_X=\int_\mathbb{R} W_1U_1+W_2U_2+\mathrm{D}^{\frac{\alpha}{2}}\partial_xW_2\mathrm{D}^{\frac{\alpha}{2}}\partial_x U_2\ \mathrm{d}x
\end{equation*}
and see that $\mathcal{L}_1$ is symmetric by
\begin{align*}
\langle \mathcal{L}_1(W_1,W_2),(U_1,U_2)\rangle_X&=\langle ((-\mathrm{D}^\alpha\partial_x^2+1)W_2,W_1),(U_1,U_2)\rangle_X\\
&=\int_\mathbb{R}(-\mathrm{D}^\alpha\partial_x^2+1)W_2U_1+W_1U_2+\mathrm{D}^{\frac{\alpha}{2}}\partial_xW_1\mathrm{D}^{\frac{\alpha}{2}}\partial_xU_2\ \mathrm{d}x\\
&=\int_\mathbb{R} W_1(-\mathrm{D}^\alpha\partial_x^2+1)U_2+W_2U_1+\mathrm{D}^{\frac{\alpha}{2}}\partial_xW_2\mathrm{D}^{\frac{\alpha}{2}}\partial_xU_1\ \mathrm{d}x\\
&=\langle (W_1,W_2),\mathcal{L}_1(U_1,U_2)\rangle_X.
\end{align*}
Since the spectrum of $L_1$ is included in $[1,\infty)$ an analog argument as in the proof of Proposition \ref{spectrum_of_L} yields that the spectrum of $\mathcal{L}_1\subset \R$ (the imaginary eigenvalues of $\mathcal{L}$ were generated by the sole negative eigenvalue of $L$, which does not occure for $L_1$). Thus
$\mathcal{L}_1$ is selfadjoint with respect to $\langle\cdot,\cdot\rangle$ and it follows immediately (see e.g. \cite[Theorem 3.16]{kato}) that
\begin{equation}\label{resolvent_est_l1_1}
\norm{(\mathcal{L}_1-\mathrm{i}n\sqrt{|\lambda|}I)^{-1}}_{X\rightarrow X}\lesssim \frac{1}{n}\qquad \mbox{for all}\quad n\in \N.
\end{equation}
Next, let us note that
\begin{align*}
\norm{\mathcal{L}_1(W_1,W_2)}_X^2=\norm{(-\mathrm{D}^\alpha\partial_x^2+1)W_2}_{L^2}^2+\norm{W_1}_{H^{1+\frac{\alpha}{2}}}^2\simeq\norm{W_1}_{H^{1+\frac{\alpha}{2}}}^2+\norm{W_2}_{H^{2+\alpha}}^2=\norm{(W_1,W_2)}_Z^2.
\end{align*}
Using this together with \eqref{resolvent_est_l1_1} we find that
\begin{gather}\label{resolvent_est_l1_2}
\begin{aligned}
\norm{(\mathcal{L}_1-\mathrm{i}n\sqrt{|\lambda|}I)^{-1}}_{X\rightarrow Z}&\simeq \norm{\mathcal{L}_1(\mathcal{L}_1-\mathrm{i}n\sqrt{|\lambda|}I)^{-1}}_{X\rightarrow X}\\
&=\norm{I+\mathrm{i}n\sqrt{|\lambda|}(\mathcal{L}_1-\mathrm{i}n\sqrt{|\lambda|}I)^{-1}}_{X\rightarrow X}\\
&\lesssim 1
\end{aligned}
\end{gather}
for all $n\in \N$, by \eqref{resolvent_est_l1_1}.
We are left to show the corresponding estimates \eqref{resolvent_est_l1_1} and \eqref{resolvent_est_l1_2} for the operator $\mathcal{L}$ if $n$ is sufficiently large. To this end, observe that
\[
(\mathcal{L}-\mathrm{i}n\sqrt{|\lambda|}I)^{-1}= (\mathcal{L}_1-\mathrm{i}n\sqrt{|\lambda|}I)^{-1}(I+\mathcal{L}_2(\mathcal{L}_1-\mathrm{i}n\sqrt{|\lambda|}I)^{-1})^{-1}.
\]
Hence, the statement is proved if we can show that
\begin{equation}\label{eq:N}
I+\mathcal{L}_2(\mathcal{L}_1-\mathrm{i}n\sqrt{|\lambda|}I)^{-1}\colon X\rightarrow X,
\end{equation}
is invertible
with $\norm{\mathcal{L}_2(\mathcal{L}_1-\mathrm{i}n\sqrt{|\lambda|}I)^{-1}}_X < 1$ for all $n\in \N$ large enough.
Since
\begin{equation*}
(\mathcal{L}_1-\mathrm{i}n\sqrt{|\lambda|}I)^{-1}=\left(\begin{array}{cc}
\mathrm{i}n\sqrt{|\lambda|}(L_1-k^2\lambda)^{-1} & 1-n^2\lambda (L_1-k^2\lambda)^{-1}\\
(L_1-n^2\lambda)^{-1} & \mathrm{i}n\sqrt{|\lambda|}(L_1-n^2\lambda)^{-1}
\end{array}\right),
\end{equation*}
and thus
\begin{equation*}
\mathcal{L}_2(\mathcal{L}_1-\mathrm{i}n\sqrt{|\lambda|}I)^{-1}W=\left(L_2[(L_1-n^2\lambda)^{-1}W_1+\mathrm{i}n\sqrt{|\lambda|}(L_1-n^2\lambda)^{-1}W_2],0\right),
\end{equation*}
we estimate
\begin{equation}\label{l2_l1_estimate}
\norm{\mathcal{L}_2(\mathcal{L}_1-\mathrm{i}n\sqrt{|\lambda|}I)^{-1}W}_X\leq \norm{L_2(L_1-n^2\lambda)^{-1}W_1}_{L^2}+|n|\sqrt{|\lambda|}\norm{L_2(L_1-k^2\lambda)^{-1}W_2}_{L^2}.
\end{equation}
Recall that $L_2=-\partial_x(c-Q_c)\partial_x-1=P-1$, where $P$ is defined in the proof of Lemma \ref{lem:spectrum_L}, and we have shown that for any $\e>0$ there exist $c(\e)>0$ such that
\begin{equation*}
\norm{Pw}_{L^2}\leq\varepsilon\norm{\mathrm{D}^{2+\alpha}w}_{L^2}+c(\varepsilon)\norm{w}_{L^2}.
\end{equation*}
Moreover, since $L_1$ is selfadjoint on $L^2(\mathbb{R})$ we have that
\begin{equation*}
\norm{(L_1-n^2\lambda)^{-1}}_{L^2\rightarrow L^2}\lesssim \frac{1}{n^2}.
\end{equation*}
We use these properties to estimate each of the two terms in the right hand side of \eqref{l2_l1_estimate} as
\begin{gather}\label{smallness_1}
\begin{aligned}
\norm{L_2(L_1-n^2\lambda)^{-1}W_1}_{L^2}&\leq \norm{P(L_1-n^2\lambda)^{-1}W_1}_{L^2}+\norm{(L_1-n^2\lambda)^{-1}W_1}_{L^2}\\
&\leq  \varepsilon \norm{\mathrm{D}^{2+\alpha}(L_1-k^2\lambda)^{-1}W_1}_{L^2}+(1+c(\varepsilon))\norm{(L_1-k^2\lambda)^{-1}W_1}_{L^2}\\
&\lesssim \varepsilon\norm{W_1}_{L^2}+\frac{1+c(\varepsilon)}{n^2}\norm{W_1}_{L^2}
\end{aligned}
\end{gather}
and similarly
\begin{gather}\label{smallness_2}
\begin{aligned}
|n|\norm{L_2(L_1-n^2\lambda)^{-1}W_2}_{L^2}&\leq |n|\norm{P(L_1-n^2\lambda)^{-1}W_2}_{L^2}+|n|\norm{(L_1-n^2\lambda)^{-1}W_2}_{L^2}\\
&\leq |n|\varepsilon\norm{\mathrm{D}^{2+\alpha}(L_1-n^2\lambda)^{-1}W_2}_{L^2}+|n|(1+c(\varepsilon))\norm{(L_1-n^2\lambda)^{-1}W_2}_{L^2}\\
&=|n|\varepsilon\norm{\frac{|\xi|^{1+\frac{\alpha}{2}}}{|\xi|^{2+\alpha}-n^2\lambda}|\xi|^{1+\frac{\alpha}{2}}\hat W_2}_{L^2}+|n|(1+c(\varepsilon))\norm{(L_1-n^2\lambda)^{-1}W_2}_{L^2}\\
&\leq |n|\varepsilon\, \text{sup}_{\xi\in[0,\infty)}\left|\frac{\xi^{1+\frac{\alpha}{2}}}{\xi^{2+\alpha}-n^2\lambda}\right|\norm{W_2}_{H^{1+\frac{\alpha}{2}}}+\frac{1+c(\varepsilon)}{|n|}\norm{W_2}_{L^2}\\
&\lesssim \varepsilon \norm{W_2}_{H^{1+\frac{\alpha}{2}}}+\frac{1+c(\varepsilon)}{|n|}\norm{W_2}_{L^2}.
\end{aligned}
\end{gather}
Hence, using \eqref{smallness_1}, \eqref{smallness_2} in \eqref{l2_l1_estimate} we find that
\begin{align*}
\norm{\mathcal{L}_2(\mathcal{L}_1-\mathrm{i}n\sqrt{|\lambda|}I)^{-1}W}_X&\lesssim \varepsilon\norm{W_1}_{L^2}+\frac{1}{n^2}\norm{W_1}_{L^2}+\varepsilon\norm{W_2}_{H^{1+\frac{\alpha}{2}}}+\frac{1+c(\e)}{n}\norm{W_2}_{L^2}\\
&\leq \left(\varepsilon+\frac{1+c(\e)}{n}\right)\norm{W}_X.
\end{align*}
By choosing $\varepsilon$ sufficiently small we obtain that $\norm{\mathcal{L}_2(\mathcal{L}_1-\mathrm{i}n\sqrt{|\lambda|}I)^{-1}}_{X\to X}<1$ for sufficiently large $n\in \N$. Thus, the operator in \eqref{eq:N} is invertible and
the desired resolvent estimates in the statement of the proposition now follow from \eqref{resolvent_est_l1_1}, \eqref{resolvent_est_l1_2}.
\end{proof}

In order to verify hypothesis (iv) of Theorem \ref{lyap} we start by determining the kernel of $L$ on the $H^{2+\alpha}_{\text{odd}}(\R)$.
\begin{lemma}\label{kerL}
	The operator $L\colon H_{\text{odd}}^{2+\alpha}(\mathbb{R})\rightarrow L_{\text{odd}}^2(\mathbb{R})$ has trivial kernel.
\end{lemma}
\begin{proof}
	Let $w\in \ker(L)$ be nontrivial. Then $Lw=0$, which implies that $\partial_x M_c\partial_x w=0$, which is equivalent to $M_c\partial_xw =d$ for some constant $d\in \R$. Since
$M_c\partial_xw\in H_{\text{odd}}^1(\mathbb{R})$, we must have that $d=0$. We deduce  that $\partial_x w\in \text{ker}(M_c)$. The kernel of $M_c$ is known to be one-dimensional and spanned by $Q_c^\prime$, see \cite[Theorem 2.3]{franklenzmann}, so we must have that $w\in \text{span}\{Q_c\}$. However, $Q_c$ is an even function and so it does not belong to $H_{\text{odd}}^{2+\alpha}(\mathbb{R})$, hence the kernel of $L$ is trivial.
\end{proof}
\begin{lemma}\label{L-inverse}
	The operator $L$ maps $H_{\text{odd}}^{2+\alpha}(\mathbb{R})$ one-to-one onto $L_{\text{odd}}^2(\R)$ and $L^{-1}$ is a bounded linear map from $L_{\text{odd}}^2(\R)$ to $H_{\text{odd}}^{2+\alpha}(\mathbb{R})$.
\end{lemma}
\begin{proof}
	We know from Lemma \ref{kerL} that $L$ is one-to-one, and as it is a self-adjoint operator we deduce from the same lemma that
	\begin{equation*}
		\{0\}=\text{ker}(L)=\text{ker}(L^*)=\text{span}(L)^\perp,
	\end{equation*}
	so $\text{codim}(\text{ran}(L))=0$, hence $L$ maps $H_{\text{odd}}^{2+\alpha}(\mathbb{R})$ onto $L_{\text{odd}}^2(\mathbb{R})$. By the open mapping theorem, $L^{-1}$ is a bounded linear operator from $L_{\text{odd}}^2(\mathbb{R})$ to $H_{\text{odd}}^{2+\alpha}(\mathbb{R})$.
\end{proof}
We are now ready to check the remaining hypotheses (i), (ii), (iii), and (iv) of Theorem \ref{lyap}.

It is immediate that $\mathcal{L}$ is densely defined and $\mathcal{L}$ is a closed operator, since the resolvent set is nonempty according to Proposition \ref{spectrum_of_L}. Thus condition (i) of Theorem \ref{lyap} holds true.
Hypothesis (ii) follows from the fact that $\mathcal{N}$ is a smooth operator from $H_{\text{odd}}^{2+\alpha}(\mathbb{R})\rightarrow H_{\text{odd}}^\alpha(\mathbb{R})$, and it is clear that $N(0)=\mathrm{d}N[0]=0$.
To show that hypothesis (iii) is satisfied we define the involution
\begin{equation*}
	S\colon \left(\begin{array}{c}
		W_1\\
		W_2
	\end{array}\right)\rightarrow \left(\begin{array}{c}
		-W_1\\
		W_2
	\end{array}\right).
\end{equation*}
We have that $S\mathcal{L}W=\mathcal{L}SW$, $S\mathcal{N}(W)=-\mathcal{N}(SW)$. So equation \eqref{dynsys1} is reversible with reverser $S$ and hence hypothesis (iii) is satisfied.
 Eventually, in order to verify the solvability condition in hypothesis (iv) we consider the equation
\begin{equation}\label{solvability-eq}
	\mathcal{L}W=-\mathcal{N}(W^*),
\end{equation}
for $W^*\in H_{\text{odd}}^{2+\alpha}(\mathbb{R})$. Equation \eqref{solvability-eq} is equivalent with
\begin{align}\label{solvability-eq-2}
	LW_2=-\frac{1}{2}\left(({W_2^*}_x)^2\right)_x,\qquad
	W_1=0.
\end{align}
Since $\left(({W_2^*}_x)^2\right)_x\in H_{\text{odd}}^\alpha(\mathbb{R})$ for any $W_2^*\in H_{\text{odd}}^{2+\alpha}(\mathbb{R})$, we directly deduce from Lemma \ref{L-inverse} that \eqref{solvability-eq-2} is solvable with solution $-\frac{1}{2}L^{-1}\left[\left(({W_2^*}_x)^2\right)_x\right]$, which is a smooth mapping from $H_{\text{odd}}^{2+\alpha}(\mathbb{R})$ to $H_{\text{odd}}^{2+\alpha}(\mathbb{R})$.

Since all the hypotheses of Theorem \ref{lyap} are fullfilled, we have proved Theorem \ref{thm:main1}.

\bigskip
\setcounter{equation}{0}
\section{Transverse instability for the fractional KP-I equation}\label{S:2}

The KP equation is derived to investigate the stability of solitary wave solutions of the KdV equation, i.e. the line solitary wave solutions of the KP equation, with respect to the perturbations which depend also to the variable in the transverse direction of the wave propagation. Therefore a natural question arises for the stability of solitary wave solutions of fKdV equation under transverse effects. The current section is devoted to the study of transverse instability for the line solitary wave solutions of the fKP-I equation.

\setcounter{equation}{0}
\subsection{Formulation of the transverse instability problem} 

First, we formulate the transverse instability problem for the fKP-I equation. Let $Q_c$ be a solitary wave solution of the fKdV equation, then it is a line solitary wave of solution of the fKP equation.
In the  moving coordinate frame $x\mapsto x-ct$, the solution $Q_c$ of the fKP-I equation becomes the stationary wave solution of
\begin{equation}\label{mf-fkp}
 (u_t-cu_x+uu_x-D^\alpha u_x)_x- u_{yy}=0.
\end{equation}
We consider a perturbation of $Q_c$ by $\psi$ as
\begin{equation}\label{pert-sol}
  u(x,y,t)=Q_c(x)+\psi(x,y,t),
\end{equation}
where $\psi(\cdot,y,t)\in L^2(\R)$ for each $(y,t)\in \R^2$ and $\psi(x,\cdot,t)\in L^\infty(\R)$ for each $(x,t)\in \R^2$.
Linearizing \eqref{mf-fkp} about $Q_c$ gives  the equation
\begin{equation}\label{pert-sol}
  (\psi_t-c\psi_x+(Q_c\psi)_x-D^\alpha \psi_x)_x-\psi_{yy}=0.
\end{equation}
Using the ansatz $\psi(x,y,t)=e^{\lambda t}\partial_xw(x,y)$ in \eqref{pert-sol}  and integrating we obtain the eigenvalue problem
\begin{equation*}
-\lambda \partial_x w=-\partial_x(D^\alpha+c-Q_c)\partial_x w - \partial_y^2 w.
\end{equation*}
Notice that the equation above is autonomous in $y$. Hence, taking the Fourier transformation with respect to $y$ we can rewrite the eigenvalue problem as
\begin{equation}\label{eq:eigenvalue_problem}
\lambda Aw_k=L(k) w_k,
\end{equation}
with
\[
A:=-\partial_x,\qquad \mbox{and}\qquad	{L}(k):=-\partial_x(D^\alpha+c-Q_c)\partial_x+ k^2, \qquad k\in \R.
\]
We thus formulate our transverse instability problem:
\begin{definition}[Linear transverse instability]\label{def}
	\emph{
The traveling wave solution $Q_c$ of the fKdV equation is \emph{linearly unstable with respect to transverse perturbations}, if there exists  $\lambda \in \mathbb{C}$ with $\mbox{{Re}}\, \lambda>0$ and
$(k,w_k)\in \R\setminus\{0\}\times H^{2+\alpha}(\R)$
solving the eigenvalue problem \eqref{eq:eigenvalue_problem}.
}
\end{definition}

\subsection{Proof of Theorem \ref{thm:main2}}

In order prove Theorem \ref{thm:main2}, i.e. to show that there exists $\lambda \in \mathbb{C}$ with $\mbox{Re}\, \lambda >0$ and $(k, w_k)\in \R \setminus\{0\}\times H^{2+\alpha}(\R)$ solving the eigenvalue problem \eqref{eq:eigenvalue_problem},
we apply a criterion formulated by { Rousset and Tzvetkov} \cite{rousset}. By Definition \ref{def} this yields the linear transverse instability result.
According to \cite{rousset}, the one dimensional solitary wave $Q_c$ is transversely unstable under the evolution of the fKP-I equation, if the following conditions are satisfied:
\begin{itemize}
\item[(A0)] For any $k\geq 0$, the operator ${L}(k):H^{2+\alpha}(\R)\to L^2(\R)$ is self-adjoint.
\item[(A1)] There exists a $k_0>0$ and an $\beta >0$ such that $L(k)\geq\beta I$ for all $|k|\geq k_0$.
\item[(A2)] $\sigma_{ess}(L(k))\subset [c_k, \infty)$ for some $c_k>0$ and $~k\neq 0$.
\item[(A3)] $L(k_1)\geq L(k_2)$ for every $k_1\geq k_2\geq 0$. In addition if for some $k>0$ and $w\neq 0$
we have $L(k)w=0$ then $\langle L'(k)w,w\rangle \geq 0$.
\item[(A4)] $L(0)$ has a simple negative eigenvalue and the rest of the spectrum is included in $[0, \infty)$.
\end{itemize}

Recall that the operator $L(0)$ equals the operator $L$ considered in Lemma \ref{lem:spectrum_L}. Therefore, we already know that $L(k)$ is a self-adjoint operator for any $k\geq 0$ by being a bounded perturbation of the self-adjoint operator $L=L(0)$. Thus (A0) is satisfied.
In order to show (A1), that is that there exists a $k_0>0$ such that $L(k)\geq \beta I$ for all $|k|\geq k_0$, we apply integration by parts and evaluate the inner product
\begin{align*}
\langle L(k)w,w\rangle = \int_{\R} (D^\alpha w_x)w_x +(c-Q_c)(w_x)^2+k^2w^2\, dx.
\end{align*}
A simple calculation shows that $\|Q_c\|_\infty\geq c$ and therefore the second term in the above integral  is not necessarily positive. Using Plancherel's identity we obtain
\begin{align*}
\langle L(k)w,w\rangle \geq \|D^\frac{\alpha}{2} w_x\|^2_{L^2}-|c-Q_c|\|w_x\|^2_{L^2} +k^2\|w\|^2_{L^2}.
\end{align*}
In the view of Lemma \ref{lemmaA}, we can estimate
\[
\|c-Q_c\|_\infty\|w_x\|^2_{L^2} \leq \e \|D^{\frac{\alpha}{2}}w_x\|^2_{L^2} + c(\e)\|w\|^2_{L^2},
\]
where $\e>0$ can be chosen small on the cost of $c(\e)$.
Hence,
\[
\langle L(k)w,w\rangle \geq (1-\e)\|D^\frac{\alpha}{2} w_x\|^2_{L^2}+\left(k^2-c(\e)\right)\|w\|^2_{L^2}.
\]
Therefore we conclude that there exists $k_0>0$, such that
\[
\langle L(k)w,w\rangle \geq \langle w,w\rangle,
\]
and condition (A1) is proved.
Since $Q_c\in H^{1+\alpha}(\R)$, the operator $K_{Q_c}:H^{2+\alpha}(\R)\to L^2(\R)$ defined by $K_{Q_c}w:=-\partial_x(Q_cw_x)$ is compact. Hence,  $L(k)$ is a compact perturbation of $L_0(k):=-\partial_x(D^\alpha+c)\partial_x +k^2$ and their essential spectra coincide, that is
\[
	\sigma_{ess}(L(k))=\sigma_{ess}(L_0(k))=[k^2,\infty).
\]
Thus, condition (A2) is satisfied.
Condition (A3) is trivially fulfilled in the view of
\[
	L^\prime(k)=2k>0\qquad \mbox{for}\quad k>0.
\]
Finally condition (A4) is already shown to hold true in Lemma \ref{lem:spectrum_L}.

\bigskip
\section{Numerical results on transverse (in)stability } \label{S:3}

Theorem \ref{thm:main2} gives an analytical result for the spectral instability of the line solitary waves under the fKP-I evolution. For the fKP-II equation, it was shown that the line solitary waves are stable \cite{mizumachi1, mizumachi2} in the case of $\alpha=2$. However, we do not have any analytical result for general $\alpha$.
In this section we perform some experiments to support the analytically obtained instability for fKP-I equation and to investigate the (in)stability for fKP-II equation numerically. 

\subsection{Numerical Method}
Considering the fractional derivatives in the equation we use a Fourier pseudo-spectral method to investigate the time evolution of the solutions of fKP equaion. The Fourier transform  of the equation \eqref{fkp-general} is
\begin{equation}\label{ft-of-fkp}
\hat{u}_t+i\left(\frac{\sigma k_y^2}{k_x}-k_x|k_x|^\alpha \right)\hat{u}+i\frac{k_x}{2}\widehat{u^2}=0,
\end{equation}
where $\hat{u}$ denotes the Fourier transform of $u$ with respect to the two-dimensional space variable $(x,y)$ and $k_x$, $k_y$ are wave numbers in the $x$ and $y$ directions respectively. To evaluate the term $i/k_x$ for $k_x=0$ we regularize it as $i/(k_x+i\lambda)$ where $\lambda=2.2\times 10^{-16}$ \cite{klein1, klein2}. To solve it numerically, we approximate the equation \eqref{ft-of-fkp} by a  discrete Fourier transform. For that aim, we assume that the solution has periodic boundary conditions on the truncated domains $x\in [-L_x, L_x]$ and $y \in  [-L_y, L_y]$. To compute the discrete Fourier transform and its inverse, we use the MATLAB functions \enquote{fft2} and \enquote{ifft2} respectively. Due to the stiffness of the resulting system of ODEs, very small time steps are required while using the explicit numerical methods. In \cite{klein3} the authors compare the performance of several fourth-order time stepping schemes for the KP-equation. They show that, implicit schemes are  very expensive computationally and the most effective methods are exponential time differencing (ETD) methods. In this study we apply the fourth-order ETD scheme proposed by Cox and Matthews \cite{cox} for the time integration of the ODE \eqref{ft-of-fkp}. 
The method is used in several studies for the KP equation \cite{klein1, klein2, klein3}.

To test the efficiency of the scheme we consider the exact line solitary wave solution
\begin{equation}\label{kdvsolution}
  u_{KdV}(x,t)=3\,c\,\mbox{sech}^2 \left(\tfrac{\sqrt{c}}{2}(x-x_0-ct)\right)
\end{equation}
for the KP equation where $\alpha=2$ in \eqref{fkp-general}. To show that the numerical scheme captures the exact solution, we investigate the time evolution of the initial data
\begin{equation*}\label{ukdv0}
   u_{KdV}(x,0)=6\,\mbox{sech} \left(\tfrac{x}{\sqrt{2}}\right),
\end{equation*}
corresponding to the initial profile of the exact solution \eqref{kdvsolution}  with $c=2$ and $x_0=0$.
Here we chose the space intervals as $x\in [-60, 60]$ and  $y\in  [-30, 30]$ with the number of grid points $N_x=2^{9}$ for the $x$ direction and $N_y=2^{7}$ for the $y$ direction.  The time interval is $t\in [0, 10]$ with time step $dt=10^{-3}$. We observe that the  $L_\infty-$norm of the difference of numerical and exact solutions is approximately of order $10^{-9}$ at time $t=10$ for both fKP-I and fKP-II equations.
We  control the change in the conserved quantity
\begin{equation*}\label{mass}
  M(u)=\int_{\mathbb{R}^2}u^2\,d(x,y)
\end{equation*}
of the fKP equation as a numerical check. We evaluate the relative conservation error
\begin{equation*}\label{cons-err}
 \mbox{error}(t)=1-\frac{M(u(t)))}{M(u(0))},
\end{equation*}
at each time step \cite{klein2}. The error  is approximately of order $10^{-11}$ at time $t=10$ for both KP-I and KP-II equations.
Even though the error in the  mass conservation  gets worse for longer time intervals, as a control of the numerical accuracy, we make sure that it is of order less than  $10^{-4}$ in all experiments.

Since we do not have the exact solutions for the general values of $\alpha$ we construct the solitary wave solutions of the fKdV equation numerically by using the Petviashvili iteration  method. This method has been widely used for fractional equations \cite{amaral, duran,le, oruc, pelinovski}. We refer to \cite{duran} and \cite{pelinovski} for details of the method while generating solitary wave solutions of the fKdV equation. In Figure \ref{fig:profiles}, we present the numerically generated line solitary wave profiles of the fKdV equation for several $\alpha$ values with wave speed $c=2$ and for various values of $c$ when $\alpha=1.5$. Here the space interval is $x\in [-100, 100]$ with the number of grid points $N_x=2^{12}$. The same values are also used for the time evolution of numerically generated waves with $y\in [-30, 30]$, $N_y=2^{7}$ and $dt=10^{-4}$. Here the aim is to make sure that the solution is at most of order
$10^{-4}$ at the boundary for the generated solitary wave, so that the necessary periodic boundary conditions for the Fourier spectral method are satisfied.

\begin{figure}[h!]
 \begin{minipage}[h]{0.45\linewidth}
   \includegraphics[scale=0.6]{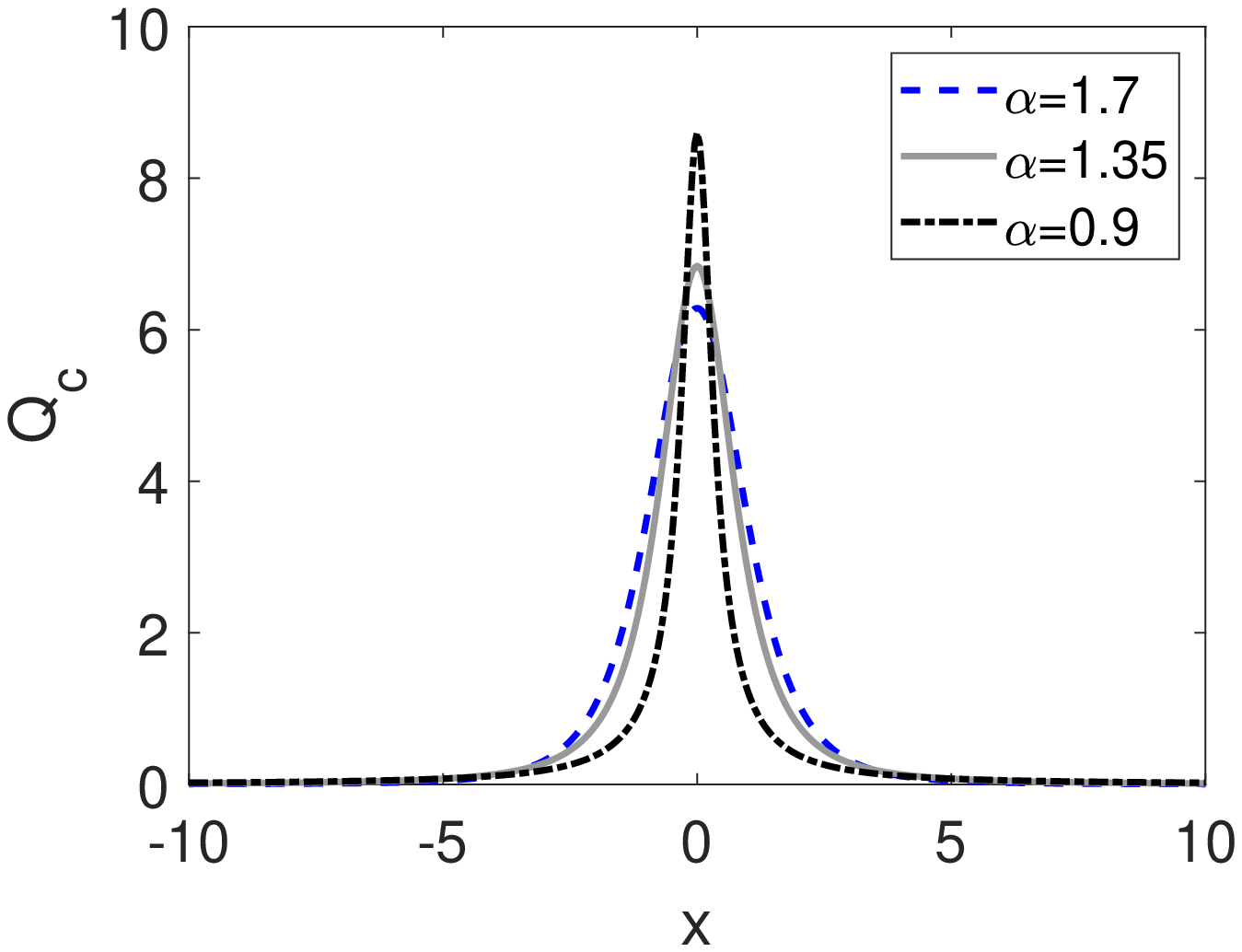}
 \end{minipage}
\hspace{30pt}
\begin{minipage}[h]{0.45\linewidth}
   \includegraphics[scale=0.6]{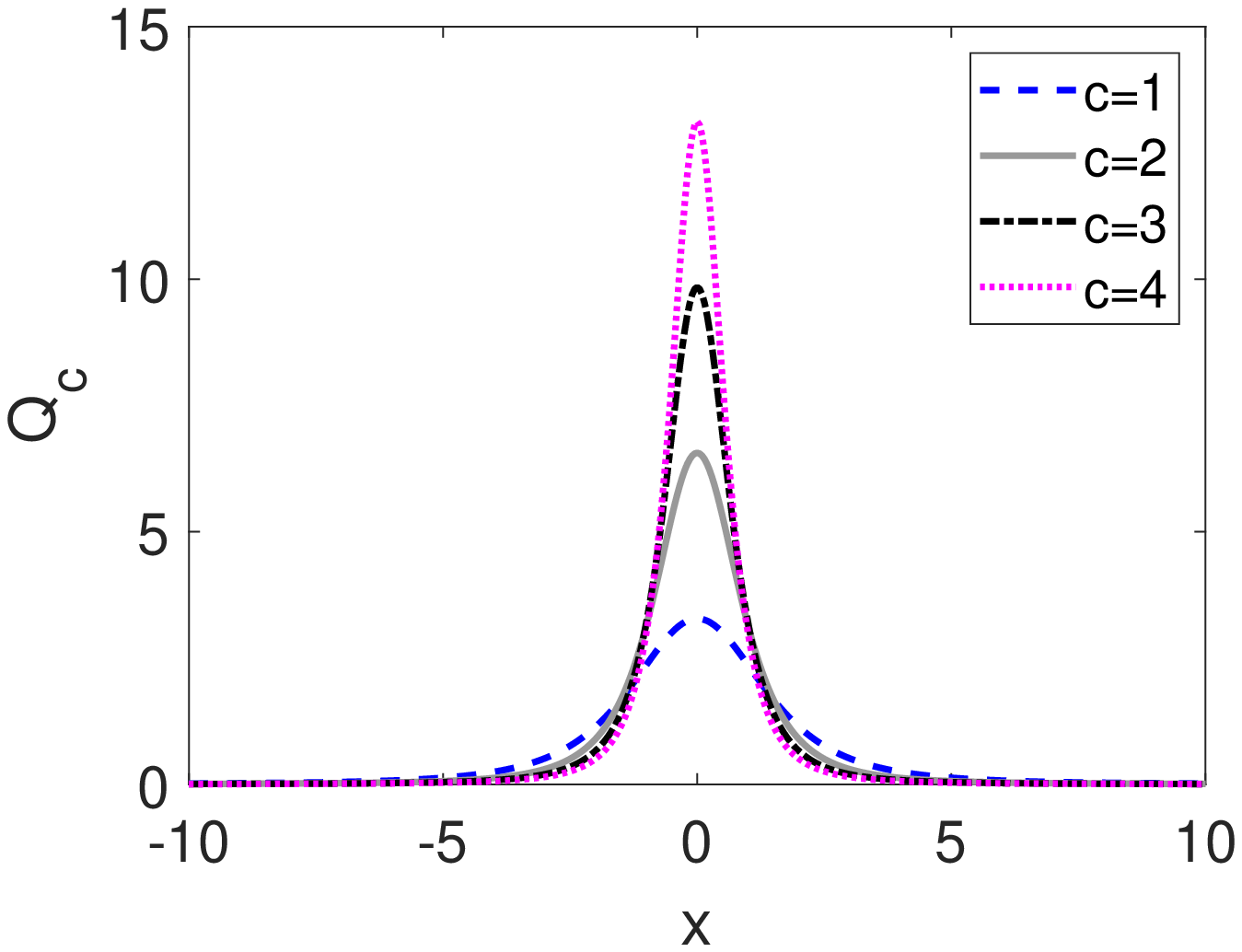}
 \end{minipage}
  \caption{Numerically generated line solitary wave profiles of fKdV equation for: $c=2$ with various values of $\alpha$ (left panel), and
  $\alpha=1.5$ with various values of $c$ (right panel).}
  \label{fig:profiles}
\end{figure}

	To investigate the transverse (in)stability of line solitary wave solution $Q_c$ of the fKP equation the natural initial condition should be
	\begin{equation*}\label{initial}
		u_0=Q_c(x)+\psi(x,y),
	\end{equation*}
	where the perturbation $\psi(x,y)$ is either localized in $x$ and $y$ or localized in $x$ and $y$-periodic \cite{klein1, molinet2}. 
For $\alpha=2$ we use \eqref{kdvsolution} at $t=0$ as the unperturbed solution $Q_c$. For the other  values of $\alpha$, $Q_c$ is the solution generated by Petviashvili iteration method. In the following experiments we use $c=2$ unless otherwise stated.
	
	For the choice of perturbation, we first recall that for the  fKP equations the zero mass constraint
	\begin{equation}\label{constraint1}
		\int_{-\infty}^{\infty} u(x,y,t)dx=0, ~~~~~~y \in\mathbb{R}, ~~~t\neq 0,
	\end{equation}
	is satisfied at any time $t>0$ even if it is not satisfied initially \cite{klein1, molinet1}. To ensure that the constraint is satisfied we impose it also on the initial data. As explained in \cite{klein2} an initial data that does not satisfy \eqref{constraint1} will yield a solution that is continuous but no longer differentiable in time. For the periodic setting in numerical implementation we impose the constraint on the initial data as
	\begin{equation*}
		\int_{-L_x}^{L_x} u_0(x,y) dx=0,
	\end{equation*}
	where $2L_x$ is the period in $x$.

\subsection{Numerical experiments for the fKP-I equation}

The aim of our numerical investigation concerning the transverse (in)stability of line solitary solutions for the fKP-I equation is twofold. On the one hand, we are going to use fully localized perturbations to demonstrate numerically the transverse instability of line solitary solutions, thereby supporting our analytical result in Theorem \ref{thm:main2}, see Subsection \ref{ss:l}. On the other hand, keeping in mind that for the classical KP-I equation there exists a critical speed $c^*=\frac{4}{\sqrt{3}}$ for which the line solitary solultion is unstable for supercritical speeds $c>c^*$, while being stable for subcritical speeds $c<c^*$ (cf. \cite{rousset3}), we aim to investigate this phenomenon for the fKP-I equation numerically. To this end, we will use a perturbation, which is localized in $x$-direction and periodic in $y$-direction, see Subsection \ref{ss:p} .

\medskip

\subsubsection{Fully localized perturbations}\label{ss:l}

Considering  the zero-mean constraint in \eqref{constraint1}, we perturb the line solitary solution by the function
\begin{equation}\label{pert}
	\psi_1(x,y)=A(x+x_0)\exp(-(x+x_0)^2-y^2),
\end{equation}
where $A$ is determined such that the maximum of the perturbation is $1/10$ of the amplitude of the unperturbed solution and $x_0$ denotes the $x$-coordinate of the location of the maximum of  the unperturbed solution.

First we consider the classical KP-I equation, where $\alpha=2$ in \eqref{fkp-general}.  In Figure \ref{KP1alpha2},  we present the evolution of the perturbed solution  at several times and the change of the $L^\infty$-norm in time. Here, we choose $x_0=10$ in \eqref{kdvsolution} and \eqref{pert} initially. 

The next experiment is on  the fractional case where $\alpha=1.7$.  Figure \ref{KP1alpha17} shows the perturbed solution of the fKP-I equation  at $t=9$. We do not display the wave at $t=0$ as they look similar for all $\alpha$ values. As above, the maximum of the wave is located at $x_0=10$ initially. A similar behaviour as in  the previous experiment is observed here. The perturbation evolves to several localized peaks which grow by time.

\begin{figure}[H]
	\begin{minipage}[t]{0.45\linewidth}
		\includegraphics[width=3.1in,height=2.3in]{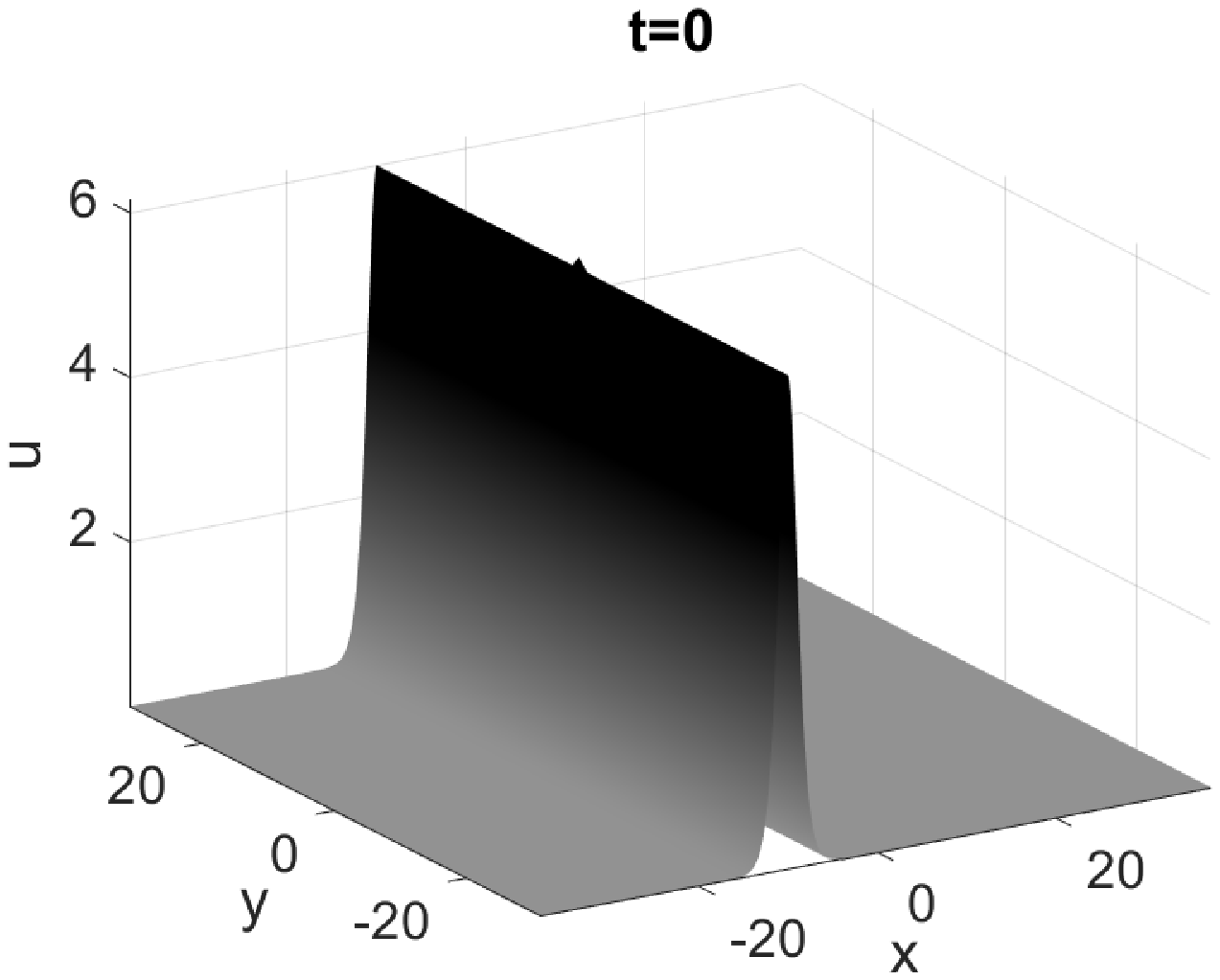}
	\end{minipage}
	\hspace{30pt}
	\begin{minipage}[t]{0.45\linewidth}
		\includegraphics[width=3.1in,height=2.3in]{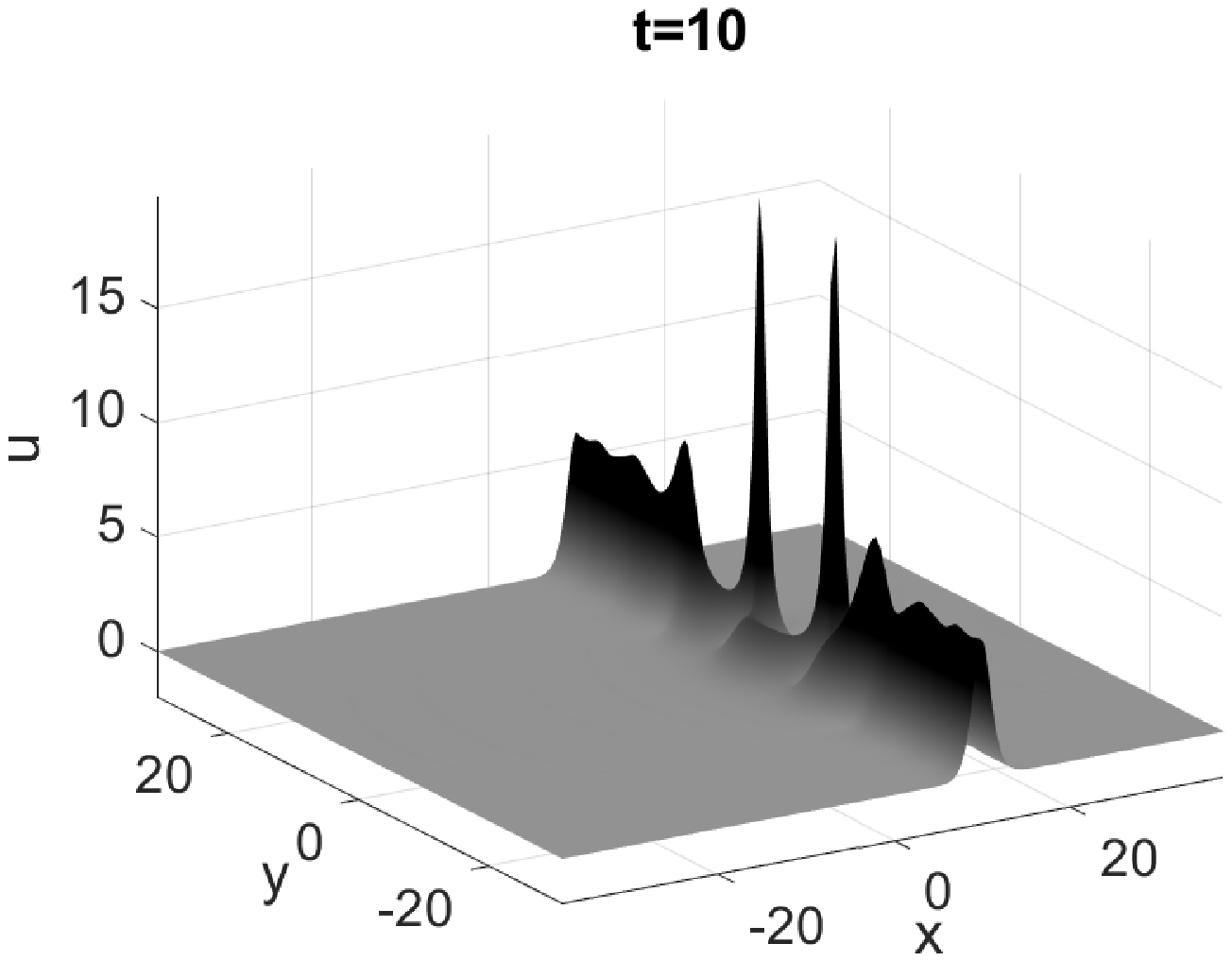}
	\end{minipage}
\end{figure}
\begin{figure}[H]
	\begin{minipage}[t]{0.45\linewidth}
		\includegraphics[width=3.1in,height=2.3in]{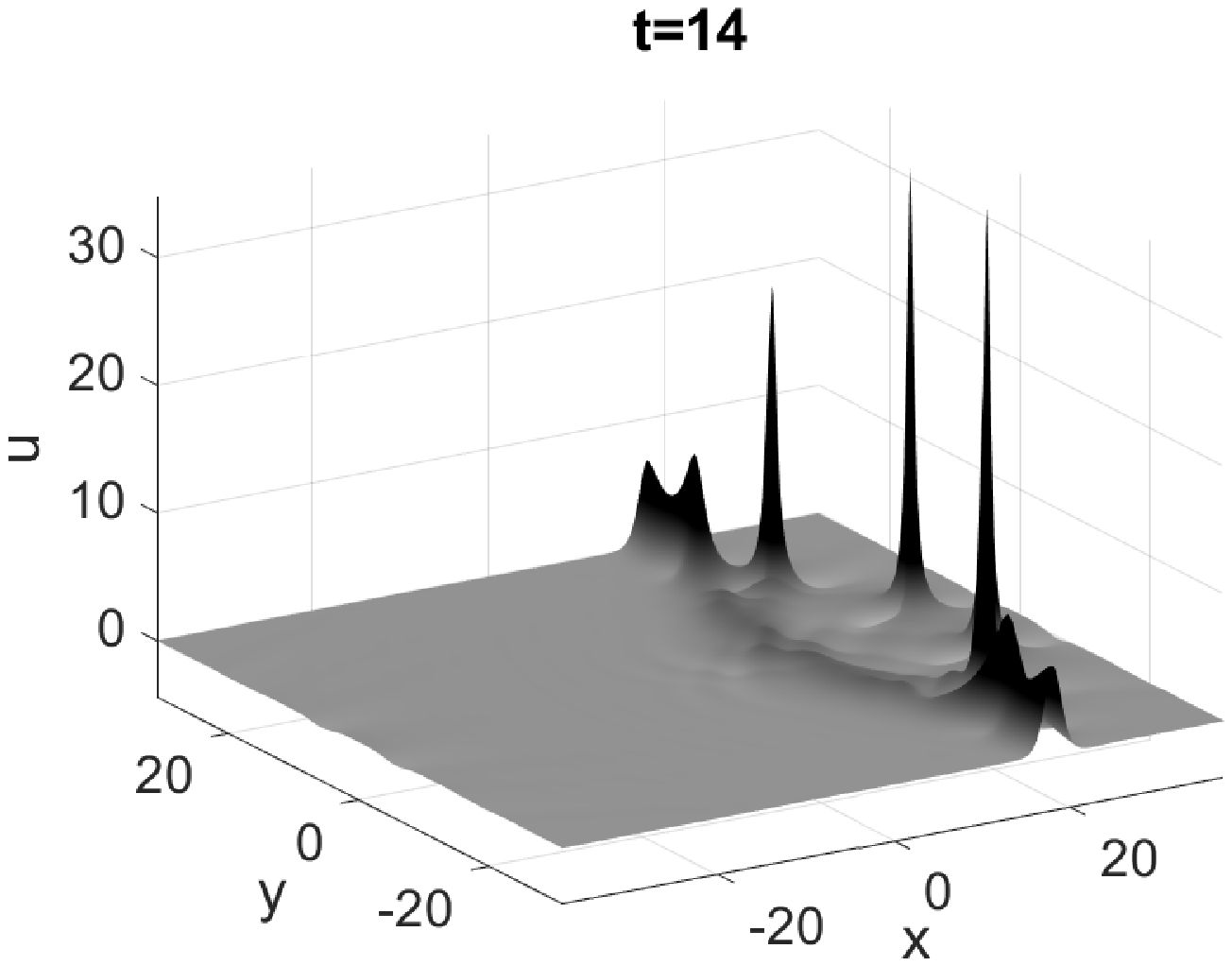}
	\end{minipage}
	\hspace{30pt}
	\begin{minipage}[t]{0.45\linewidth}
		\includegraphics[width=3.1in,height=2.3in]{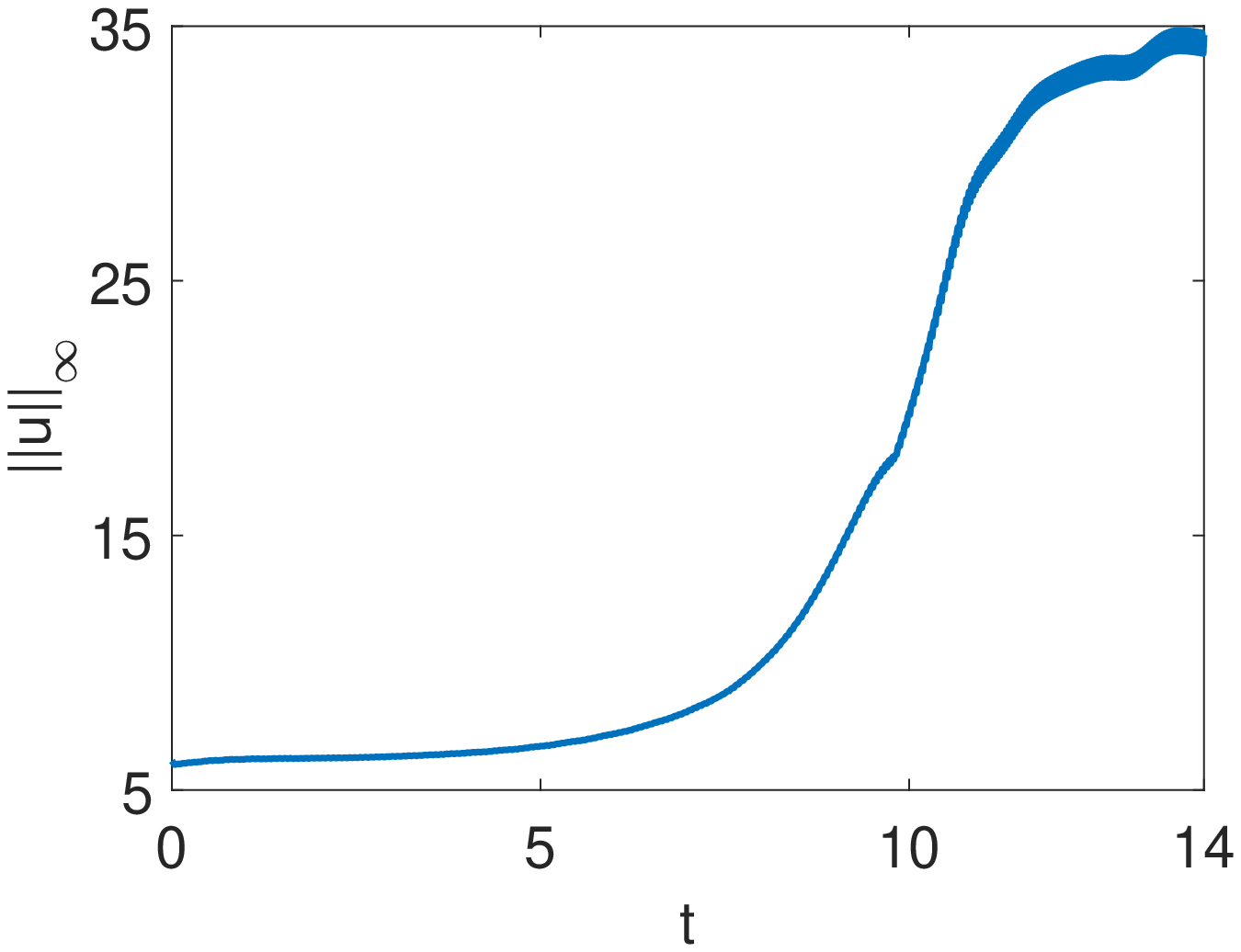}
	\end{minipage}
	\caption{Time evolution of perturbed solution of the fKP-I equation when $\alpha=2$, $c=2$ and the change of the $L^\infty$-norm in time.}
	\label{KP1alpha2}
\end{figure}

 \begin{figure}[H]
	\begin{minipage}[t]{0.45\linewidth}
		\includegraphics[width=3.1in,height=2.3in]{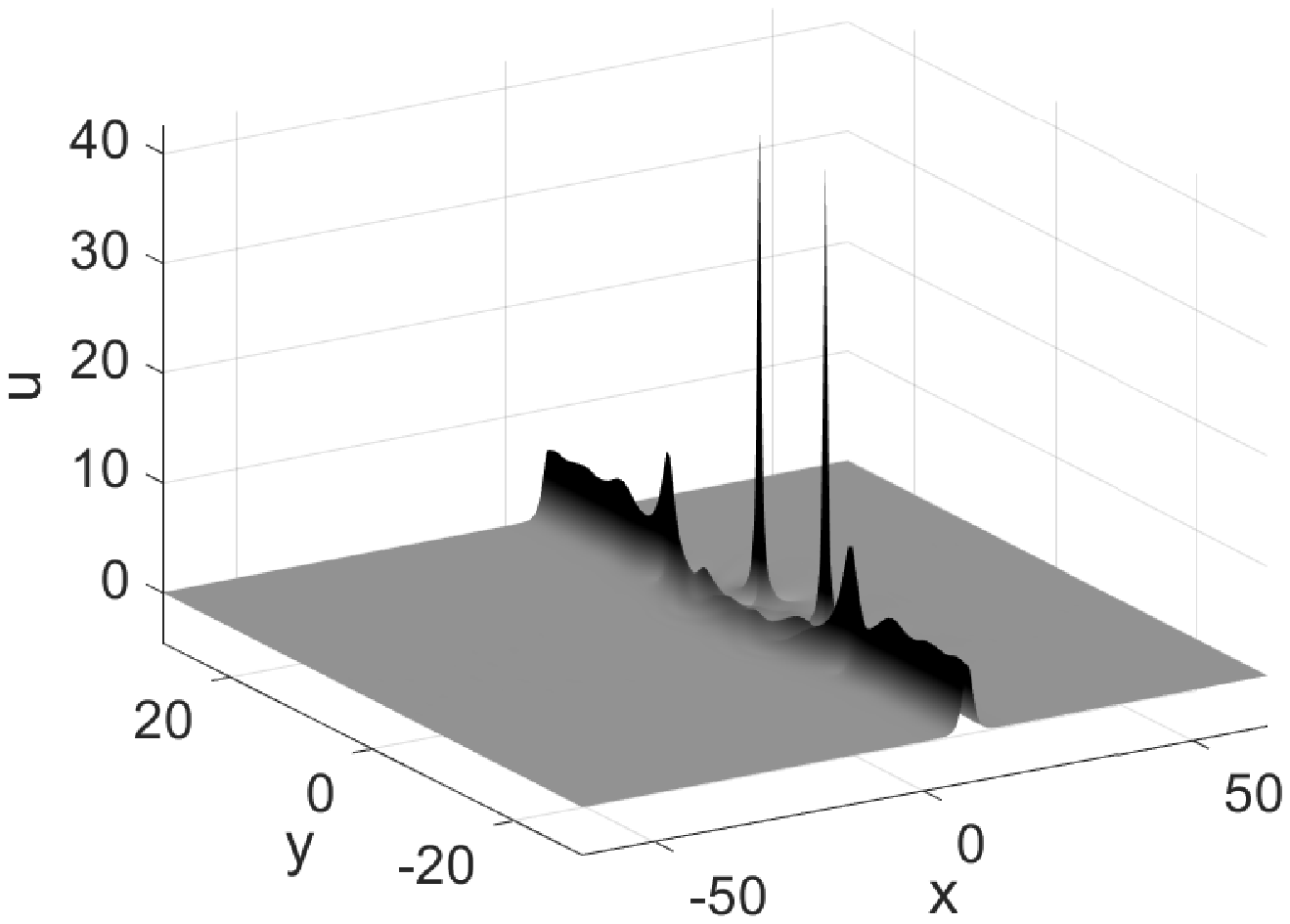}
	\end{minipage}
	\hspace{30pt}
	\begin{minipage}[t]{0.45\linewidth}
		\includegraphics[width=3.1in,height=2.3in]{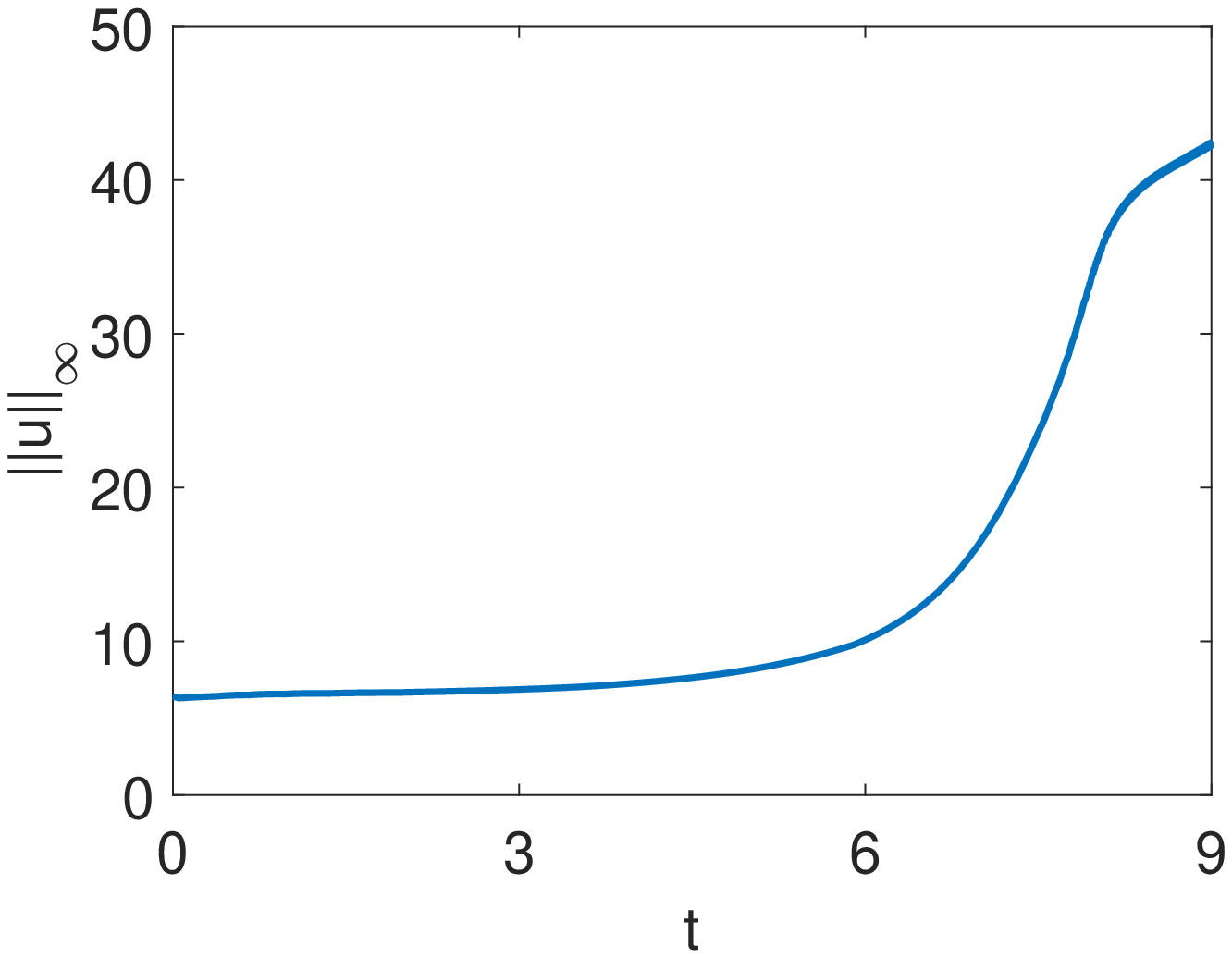}
	\end{minipage}
	\caption{Perturbed solution of the fKP-I equation at $t=9$ when $\alpha=1.7$, $c=2$ (left panel)
		and the change of the $L^\infty$-norm in time (right panel).}
	\label{KP1alpha17}
\end{figure}

Figure \ref{KP1alpha135}  depicts the wave for the fKP-I equation with $\alpha=1.35$. Here the $\alpha$ value is just above the $L^2$-critical value $\alpha=4/3$. We then present a case where $\alpha$ is $L^2$-supercritical. Figure \ref{KP1alpha09} shows the fKP-I equation with $\alpha=0.9$.

\begin{figure}[H]
	\begin{minipage}[t]{0.45\linewidth}
		\includegraphics[width=3.1in,height=2.3in]{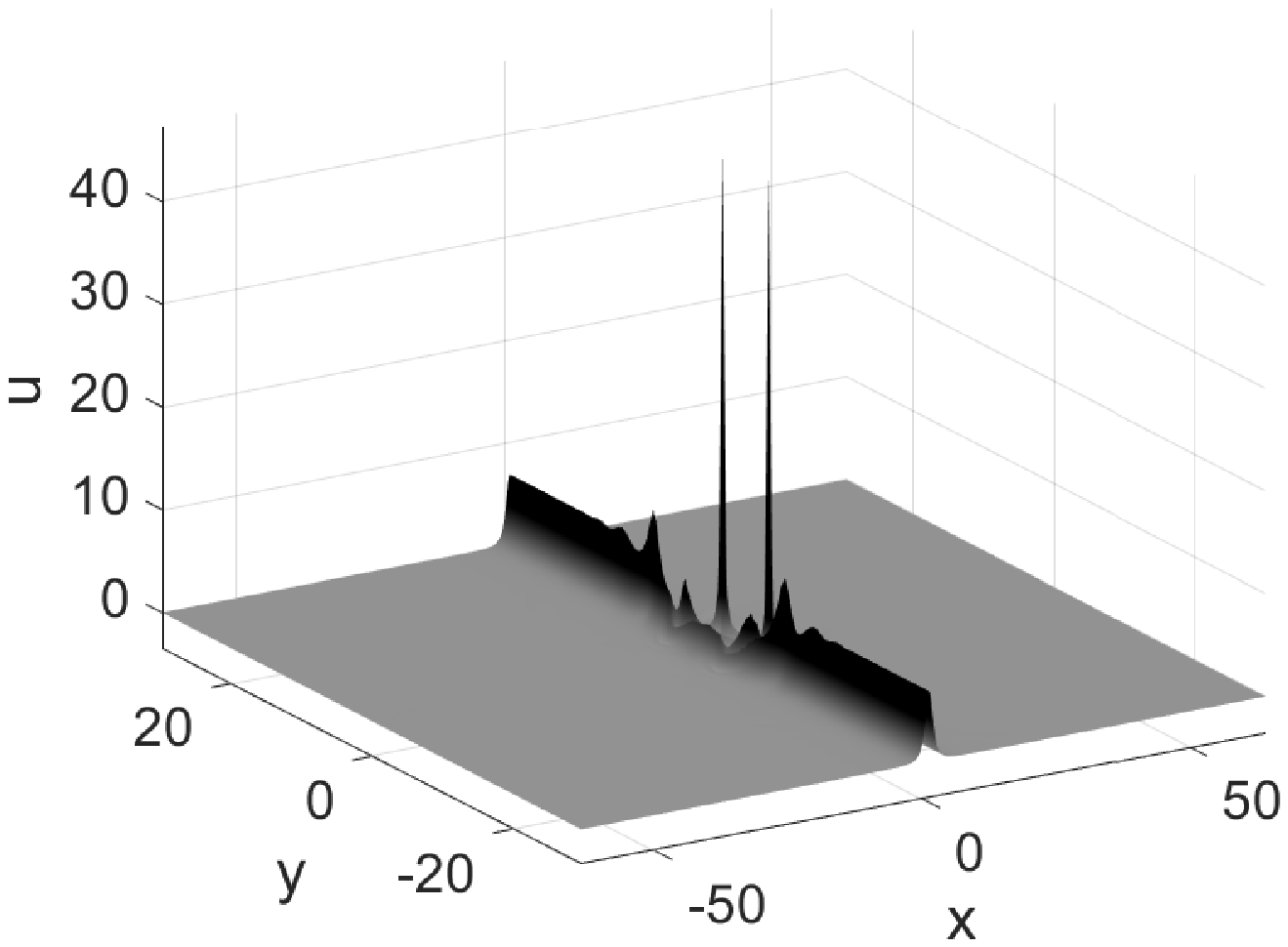}
	\end{minipage}
	\hspace{30pt}
	\begin{minipage}[t]{0.45\linewidth}
		\includegraphics[width=3.1in,height=2.3in]{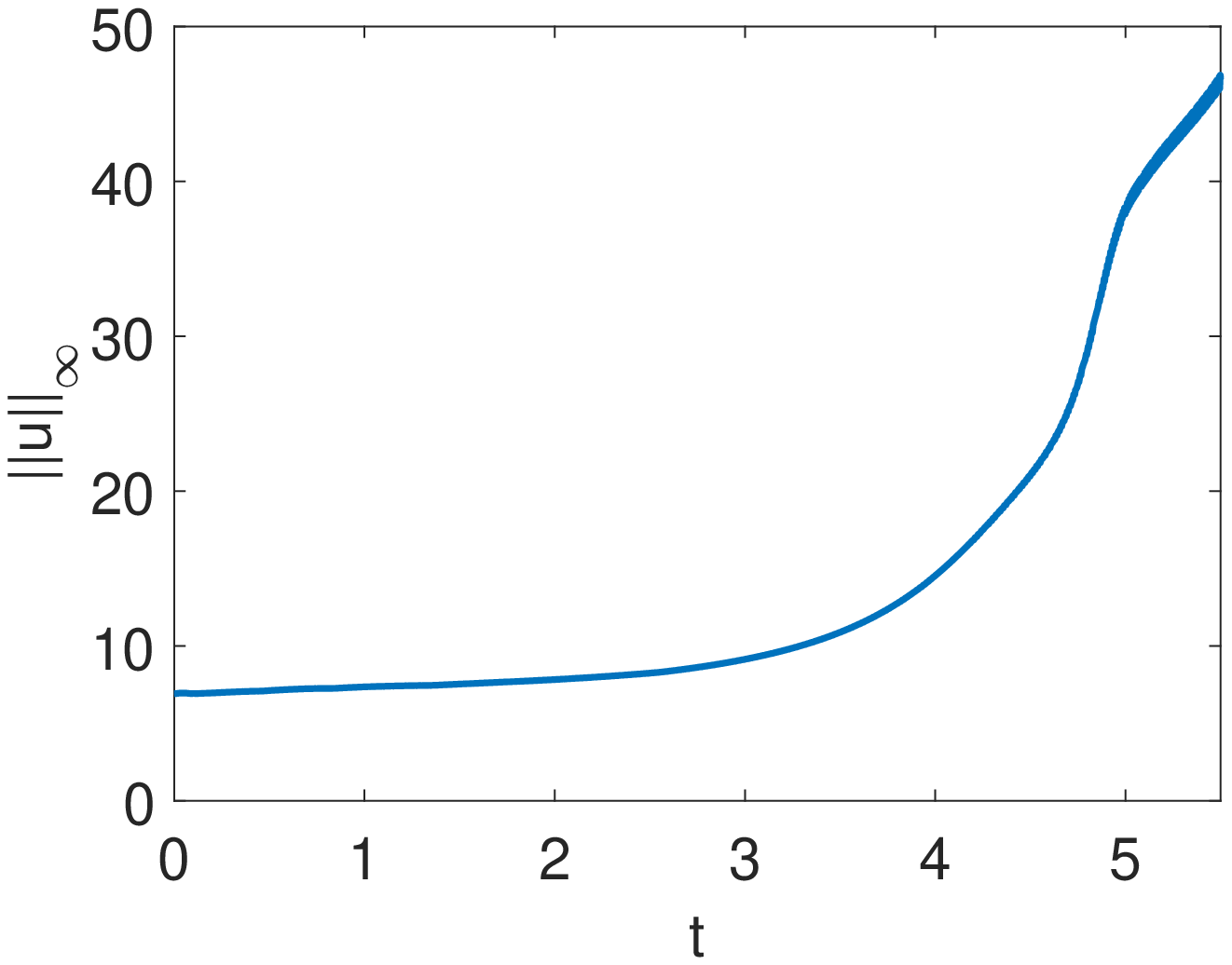}
	\end{minipage}
	\caption{Perturbed solution of the fKP-I equation at $t=5.5$ when $\alpha=1.35$, $c=2$ (left panel) and the change of the $L^\infty$-norm in time (right panel).}
	\label{KP1alpha135}
\end{figure}

\begin{figure}[H]
	\begin{minipage}[t]{0.45\linewidth}
		\includegraphics[width=3.1in,height=2.3in]{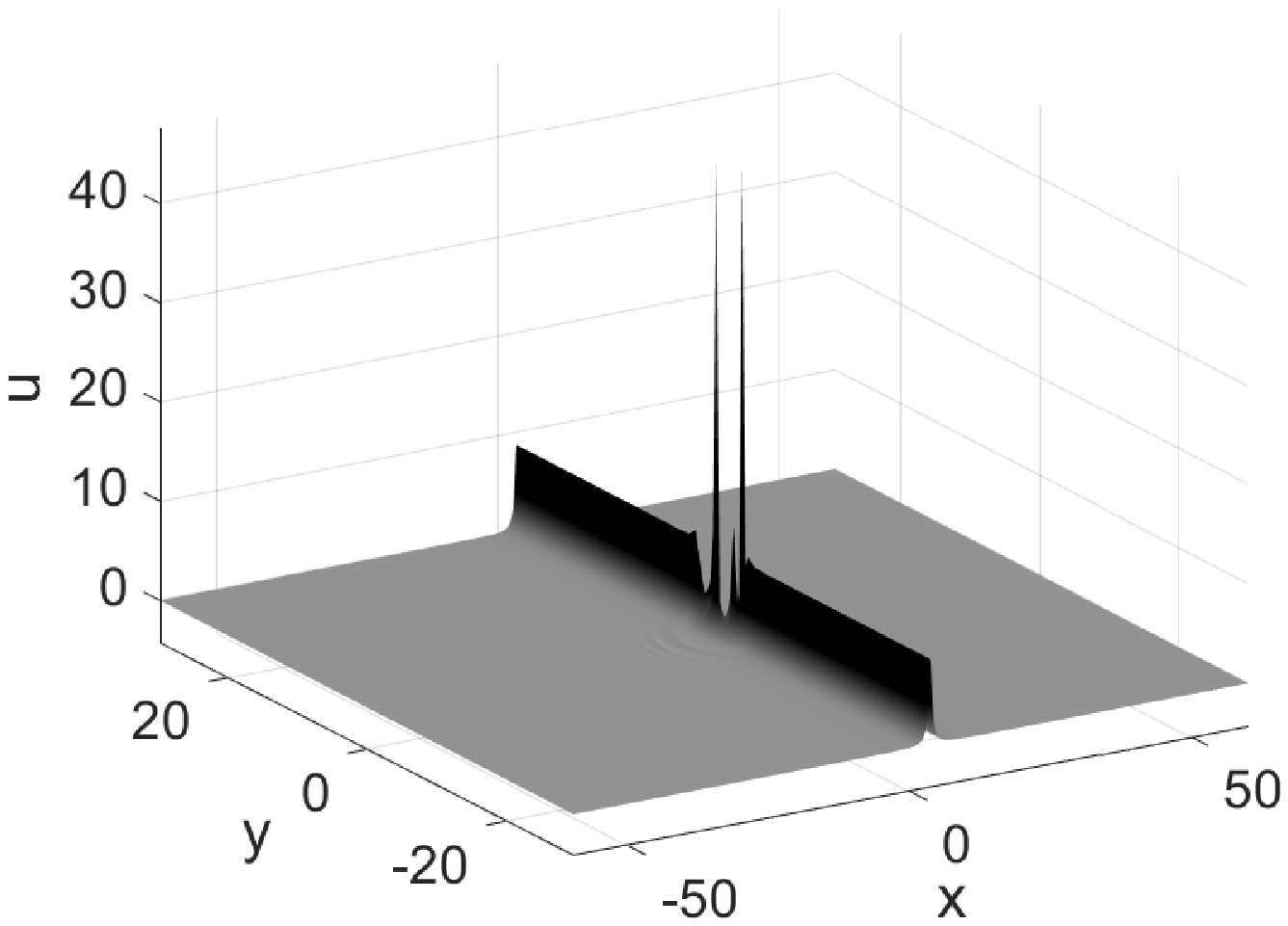}
	\end{minipage}
	\hspace{30pt}
	\begin{minipage}[t]{0.45\linewidth}
		\includegraphics[width=3.1in,height=2.3in]{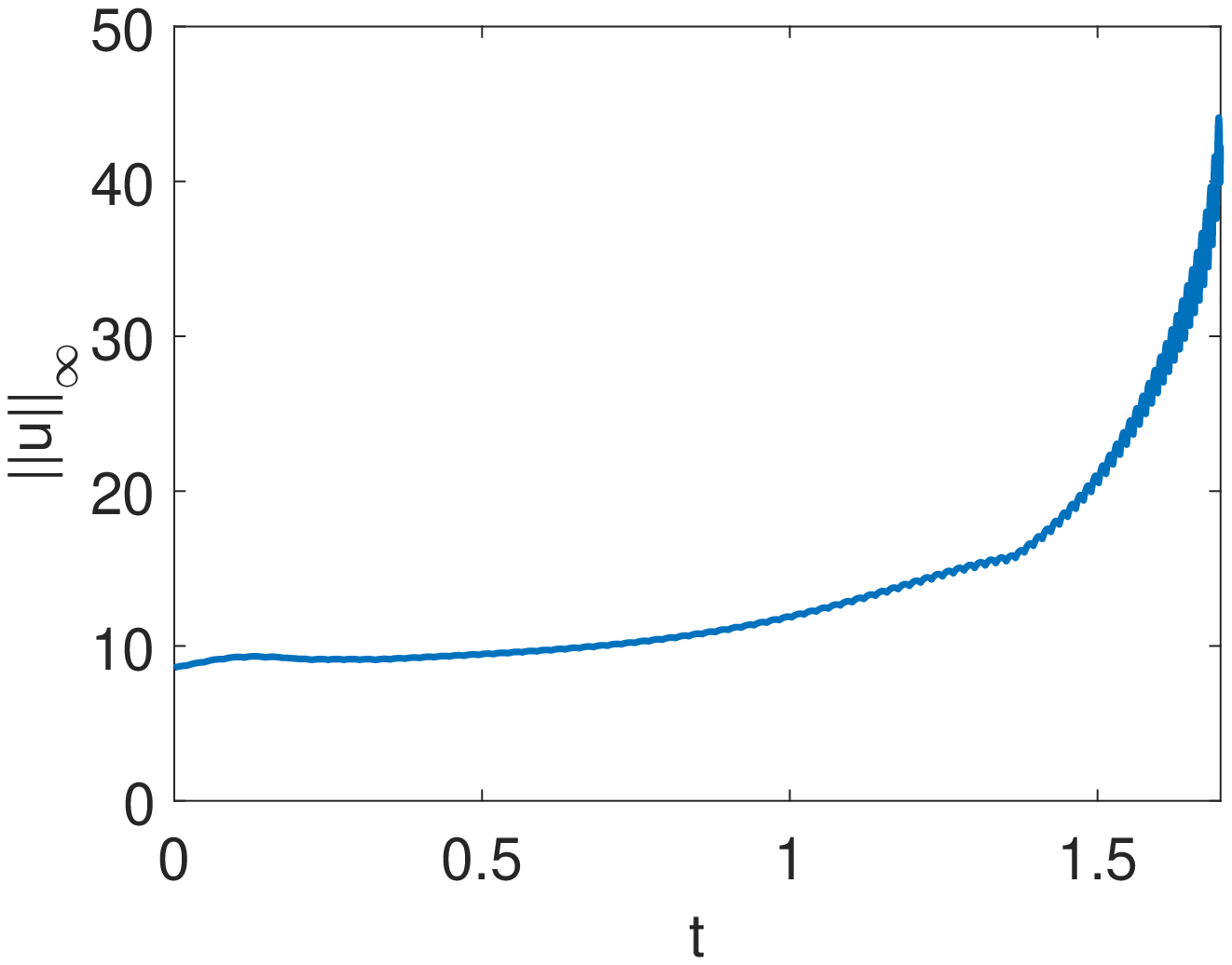}
	\end{minipage}
	\caption{Perturbed solution of the fKP-I equation  at $t=1.7$ when $\alpha=0.9$, $c=2$ (left panel) and the change of the $L^\infty$-norm in time (right panel).}
	\label{KP1alpha09}
\end{figure}

Theorem \ref{thm:main2} states the linear transverse instability of line solitary wave solution under the fKP-I flow when $\frac{1}{3}< \alpha <2$.   The numerical experiments clearly support the analytically obtained instability result.
In all examples  the initial perturbation evolves into localized peaks which also move in time. We can see that the peaks move faster than the wave itself for all $\alpha$ values. Even though it is not very visible for $\alpha=0.9$ we can confirm that the location of the peaks are ahead of the line solitary wave.
For $\alpha=2$ and  $\alpha=1.75$ we observe that the $L^\infty$-norm increases and converges to a plateau level after some time. This time may be considered where the peaks start to develop a lump formation. Similar behaviour has been observed for the KP-I equation in \cite{klein1}. For smaller $\alpha$ values  the error in the conserved quantity increases by the increasing time in the numerical experiments and we are not able to perform the numerical experiments for larger time values. Therefore the lump-formation  is not observed for small $\alpha$ values but the increasing  $L^\infty$-norm indicates a strong instability.

The next experiment is to understand  the relation between the \enquote{blow-up} time with both $\alpha$ and $c$. Figure \ref{time} shows the times where the $L^{\infty}$-norm of the solution is twice as large as the $L^{\infty}$-norm of the initial wave. We first fix $c=2$ and consider several values of $\alpha$. Then we fix $\alpha=1.5$ and consider various $c$ values. Here we note that choosing  same $A$ in \eqref{pert} for every $\alpha$ and $c$ slightly changes the graphs. 

\begin{figure}[H]
	\begin{minipage}[t]{0.45\linewidth}
		\includegraphics[width=3.1in,height=2.3in]{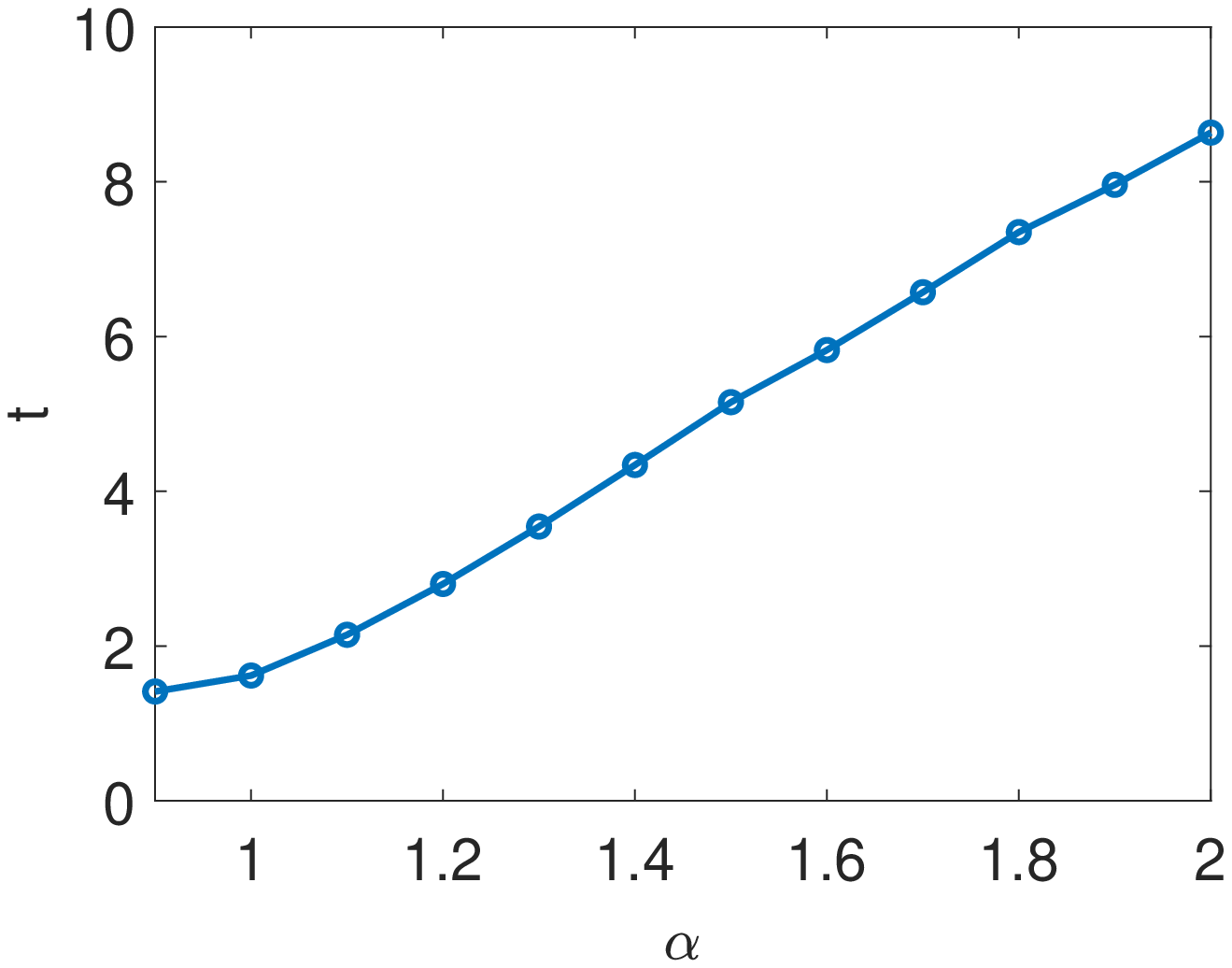}
	\end{minipage}
	\hspace{30pt}
	\begin{minipage}[t]{0.45\linewidth}
		\includegraphics[width=3.1in,height=2.3in]{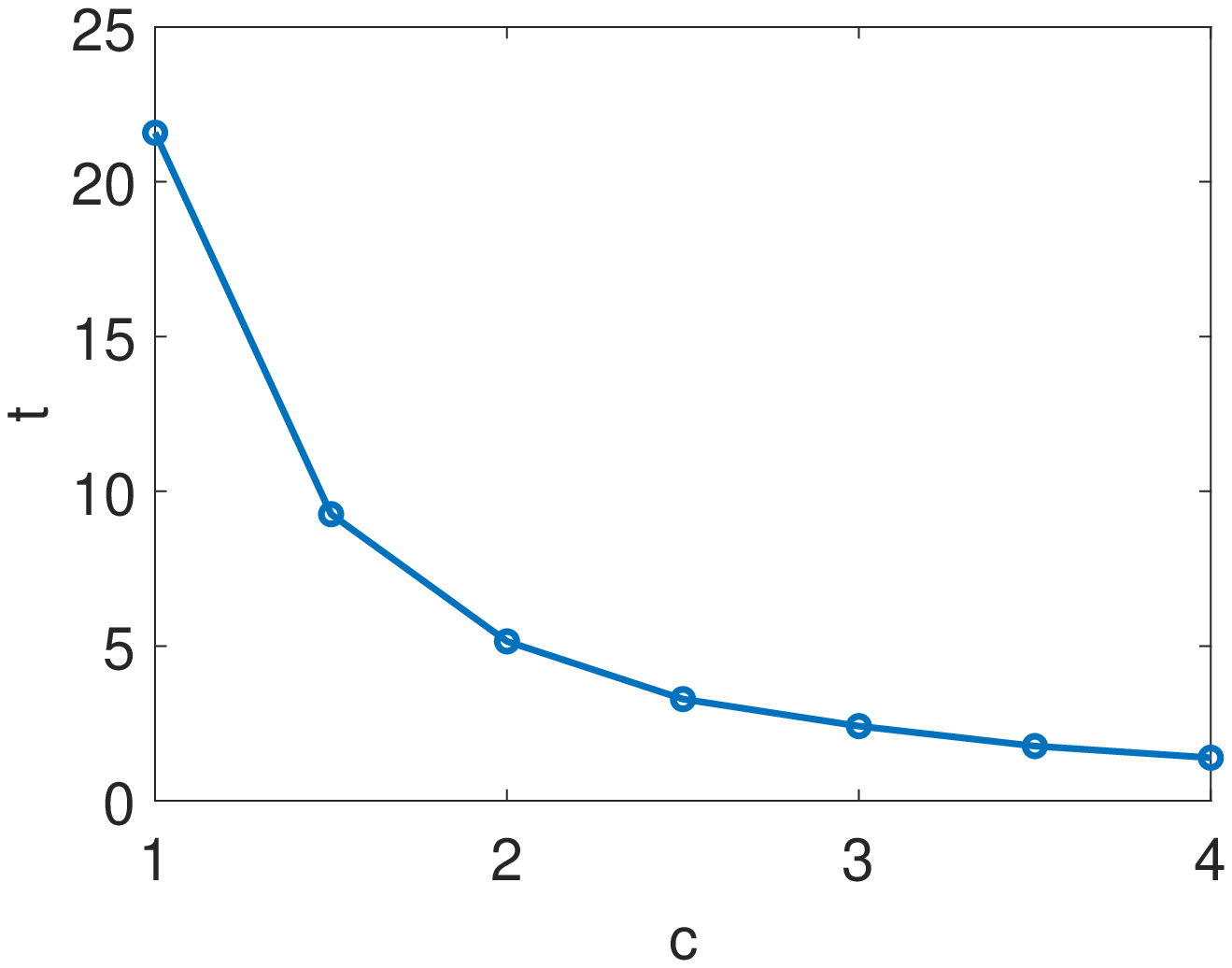}
	\end{minipage}
	\caption{The time that the amplitude of the perturbed solution doubles for various $\alpha$ values when $c=2$ (left panel) and for various values of $c$ when $\alpha=1.5$
		(right panel).}
	\label{time}
\end{figure}

 \medskip

\subsubsection{Partially  localized perturbations}\label{ss:p}  In \cite{rousset3}, Rousset and  Tzvetkov prove that the KdV solitary wave with a subcritical speed $0< c < c^*$ is orbitally stable under the global KP-I flow and orbitally unstable with a  supercritical speed $c > c^*$ in case of perturbations which are periodic in $y$ direction. For the KP-I equation  the critical speed is $c^*=4/\sqrt{3}$.  We perform several numerical experiments to observe this phenomena  for the KP-I equation, where $\alpha=2$. In order to understand  if a similar  behaviour occurs for the fKP-I equation we consider the case where $\alpha=1.5$. For this aim, we perturb the line solitary solutions by the function
\begin{equation}\label{pert2}
  \psi_2(x,y)=A(x+x_0)\exp(-(x+x_0)^2)\cos(y).
\end{equation}
Here $A$ and $x_0$ are determined as in \eqref{pert}.

Figures \ref{supercritical} and \ref{subcritical} depict the cases for supercritical and subcritical speeds $c_1=c^*+0.1$ and \mbox{$ c_2=c^*-0.1$ } respectively for KP-I equation. In both figures $x_0=20$. We see that the $L^\infty$-norm increases very fast for the supercritical speed. On the other hand the $L^\infty$-norm is oscillating and the oscillations are decreasing by the time for the subcritical speed. Therefore the numerical results indicate a nonlinear instability when $c > c^*$ and a nonlinear stability when $0< c < c^*$, which are  compatible with theoretical result in \cite{rousset3}.

 \begin{figure}[H]
  \begin{minipage}[b!]{0.45\linewidth}
    \includegraphics[width=3.1in,height=2.3in]{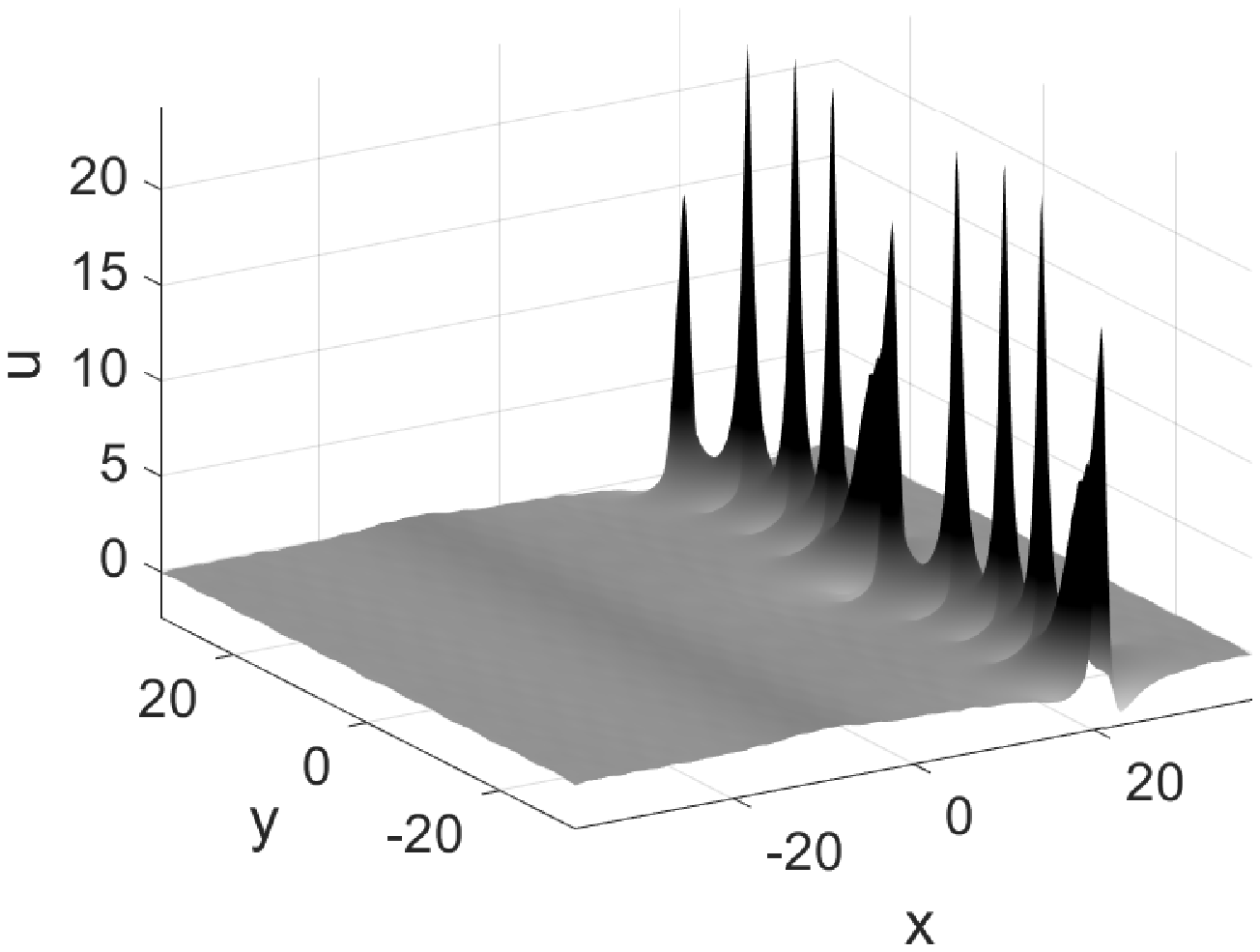}
  \end{minipage}
 \hspace{30pt}
 \begin{minipage}[h!]{0.45\linewidth}
    \includegraphics[width=3.1in,height=2.3in]{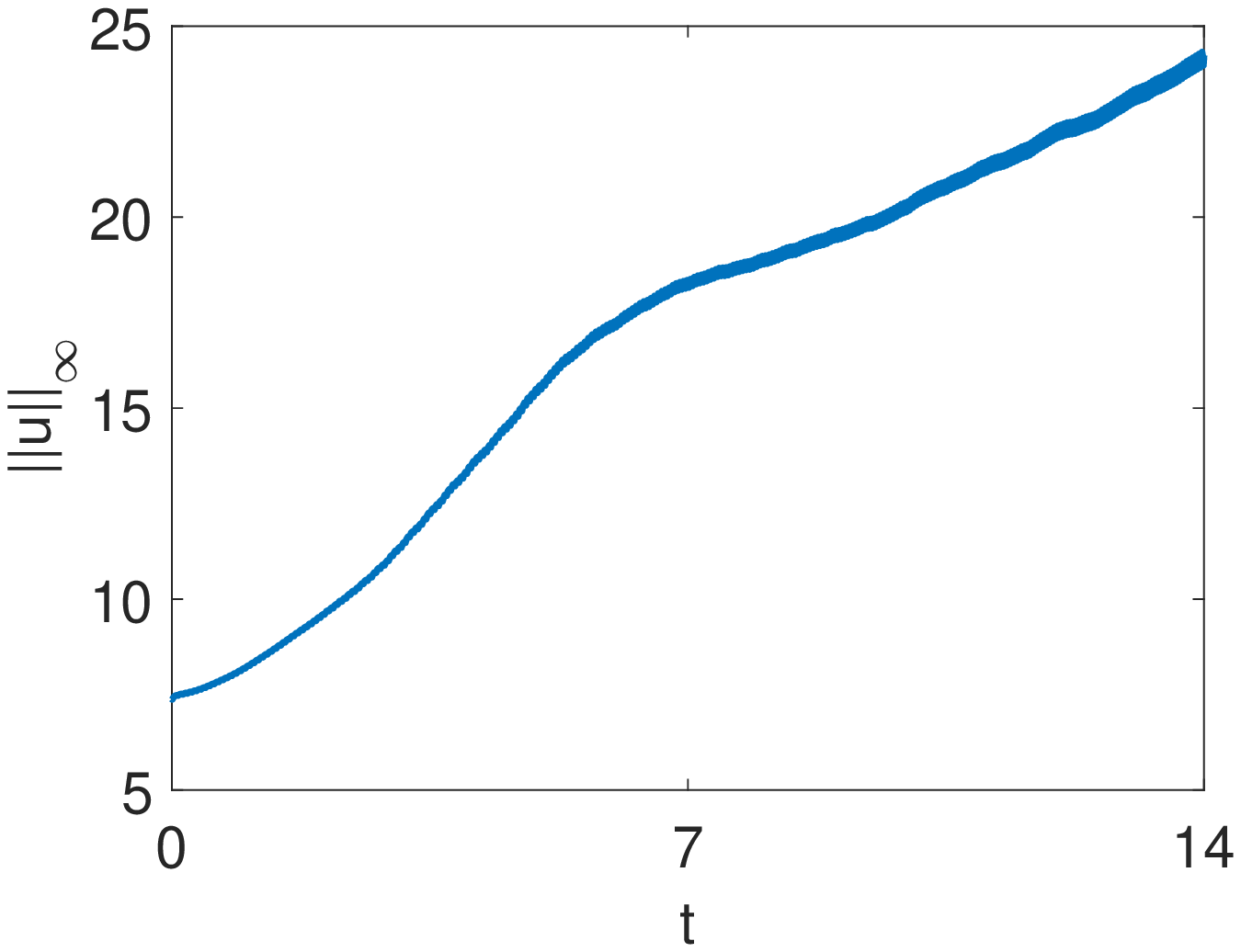}
  \end{minipage}
  \caption{The solution of the fKP-I equation that is perturbed with a nonlocalized perturbation at $t=14$ for $\alpha=2$ and $c=c^*+0.1$ (left panel) and the change of the $L^\infty$-norm in time (right panel).}
   \label{supercritical}
 \end{figure}

  \begin{figure}[H]
  \begin{minipage}[b!]{0.45\linewidth}
    \includegraphics[width=3.1in,height=2.3in]{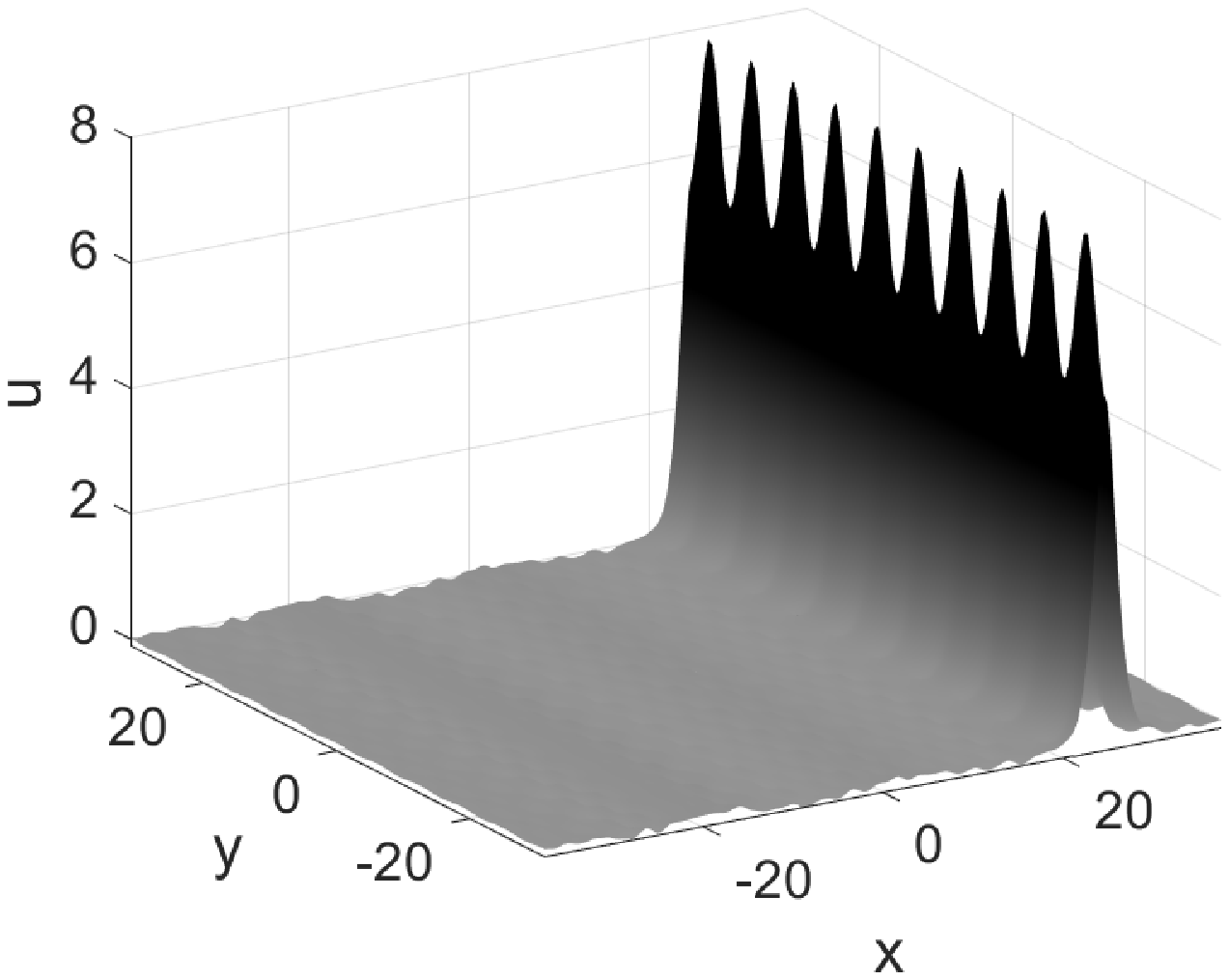}
  \end{minipage}
 \hspace{30pt}
 \begin{minipage}[h!]{0.45\linewidth}
    \includegraphics[width=3.1in,height=2.3in]{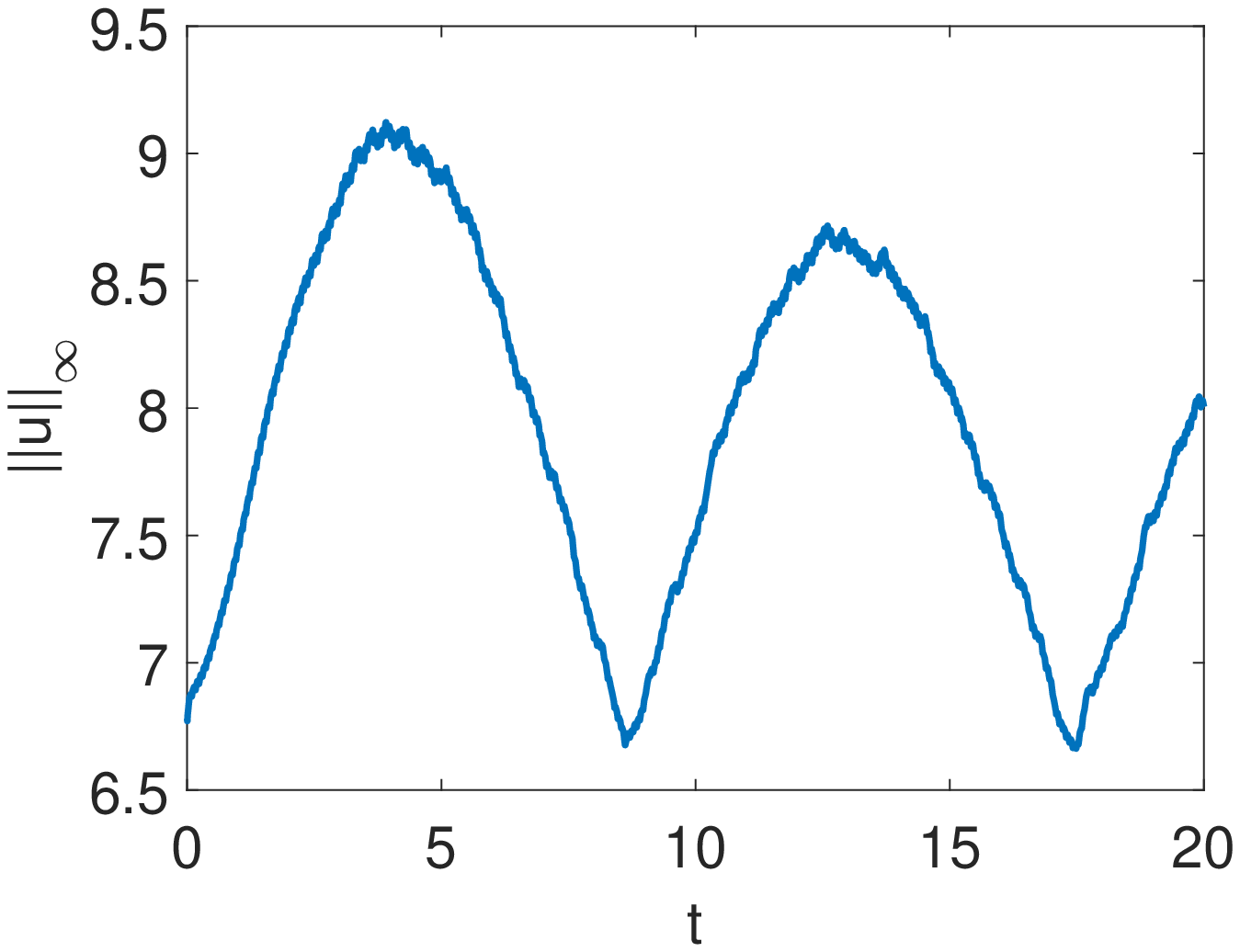}
  \end{minipage}
  \caption{The solution of the fKP-I equation that is perturbed with nonlocalized perturbation at $t=20$ for $\alpha=2$ and $c=c^*-0.1$ (left panel) and the change of the $L^\infty$-norm in time (right panel).}
   \label{subcritical}
 \end{figure}

  \begin{figure}[H]
	\begin{minipage}[htb!]{0.45\linewidth}
		\includegraphics[width=3.1in,height=2.3in]{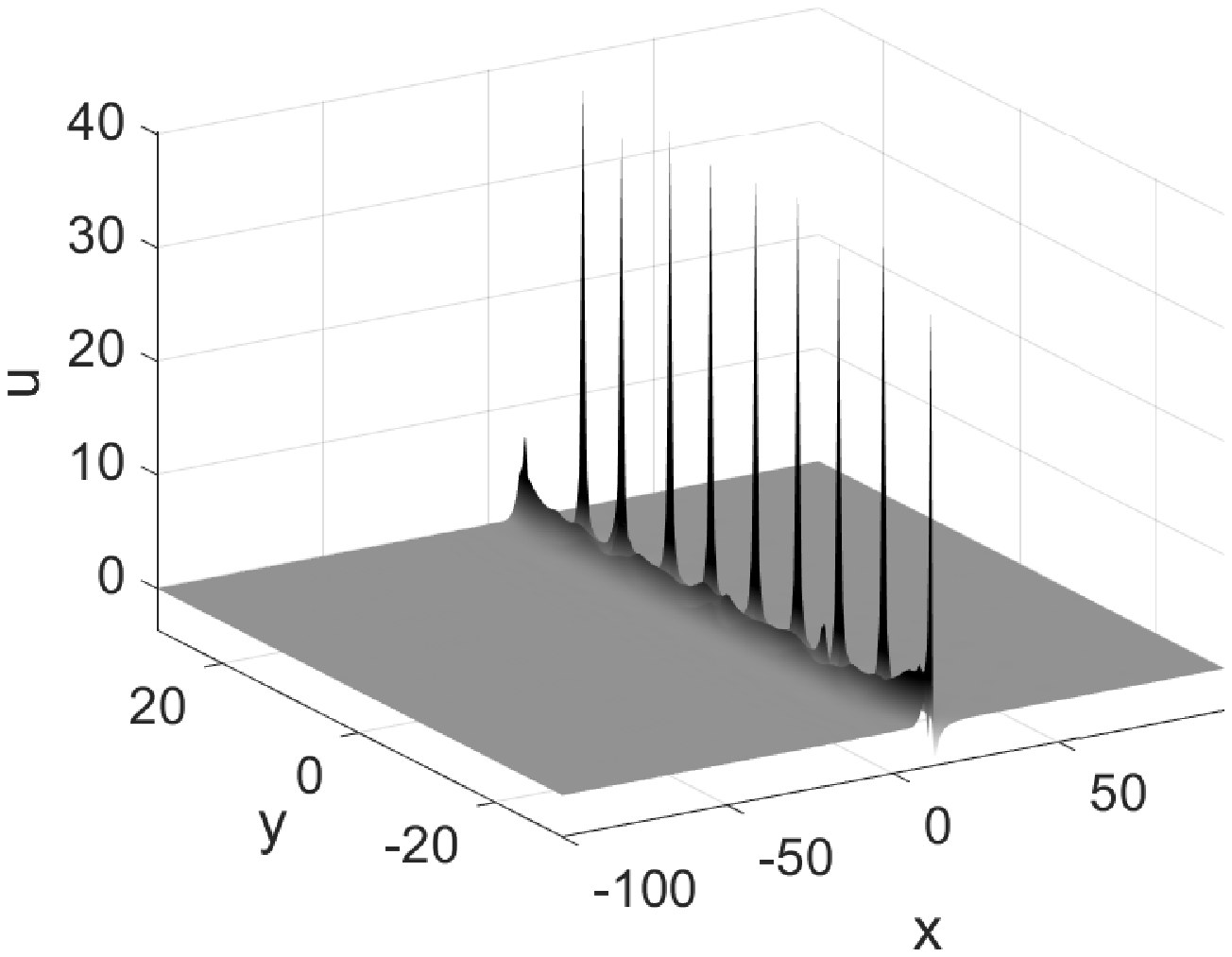}
	\end{minipage}
	\hspace{30pt}
	\begin{minipage}[htb!]{0.45\linewidth}
		\includegraphics[width=3.1in,height=2.3in]{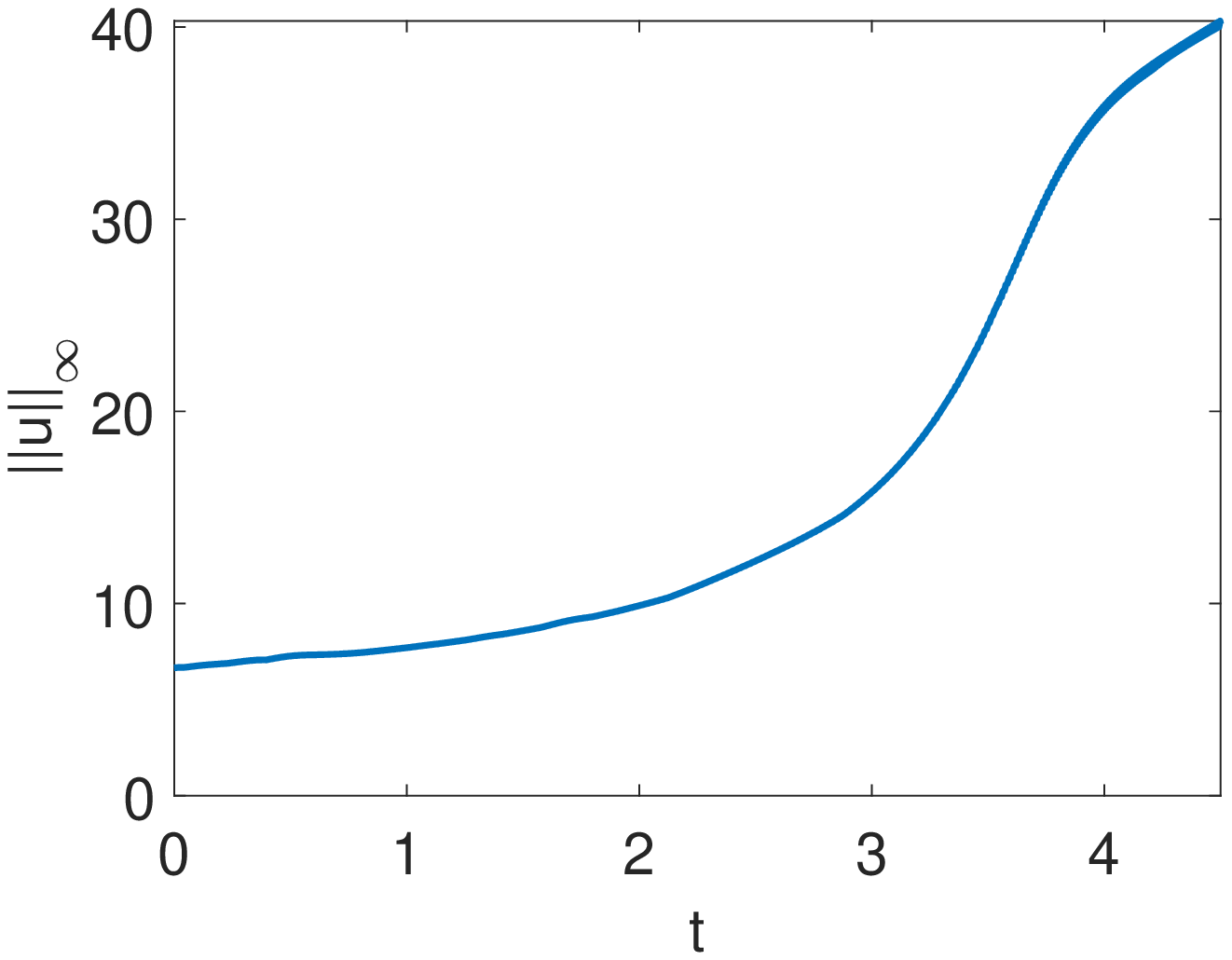}
	\end{minipage}
	\caption{The solution of the fKP-I equation that is perturbed with nonlocalized perturbation at $t=20$ for $\alpha=1.5$ and $c=2$ (left panel) and the change of the $L^\infty$-norm in time (right panel).}
	\label{P2alpha15instabil}
\end{figure}

\begin{figure}[H]
	\begin{minipage}[htb!]{0.45\linewidth}
		\includegraphics[width=3.1in,height=2.3in]{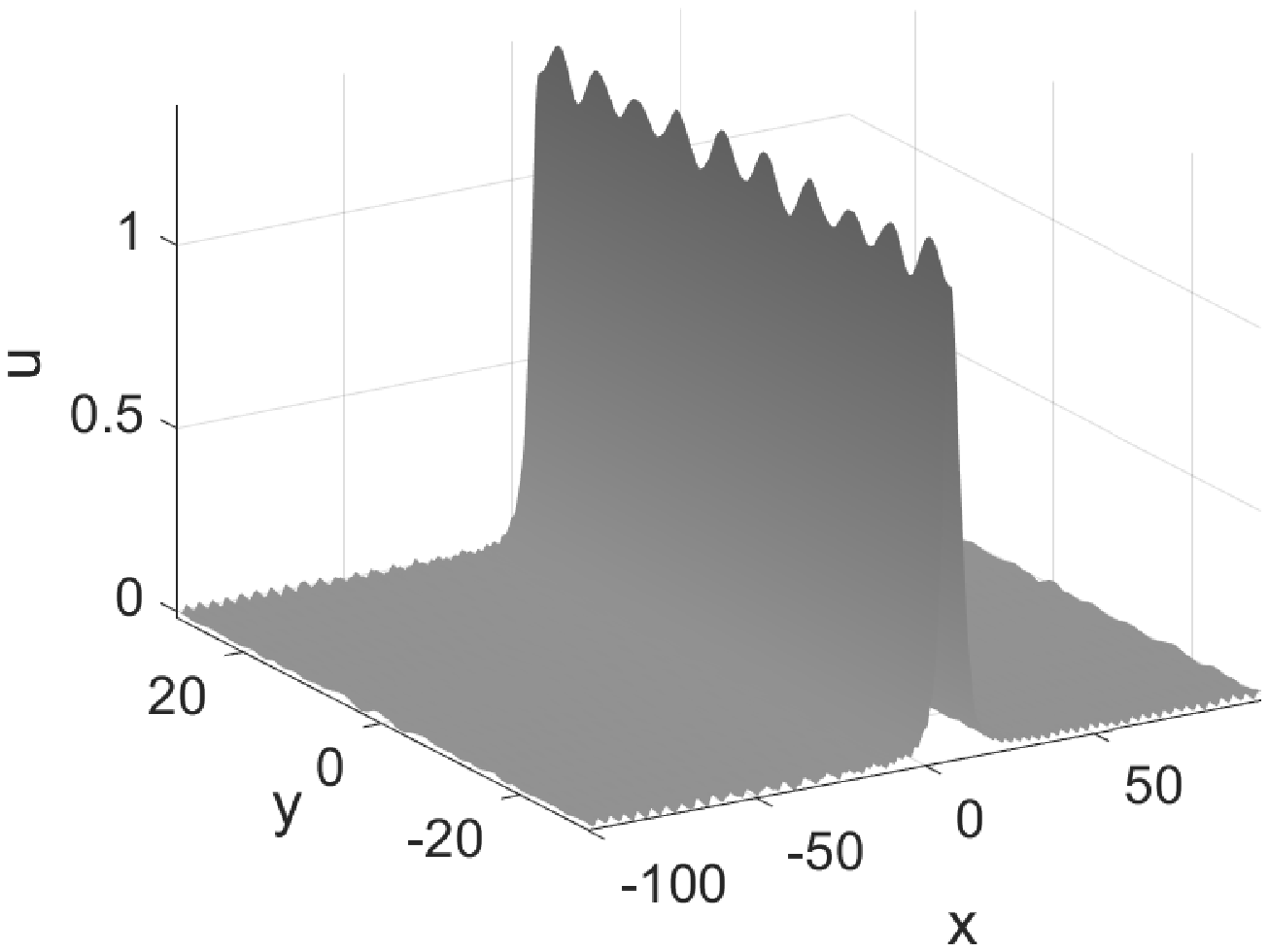}
	\end{minipage}
	\hspace{30pt}
	\begin{minipage}[htb!]{0.45\linewidth}
		\includegraphics[width=3.1in,height=2.3in]{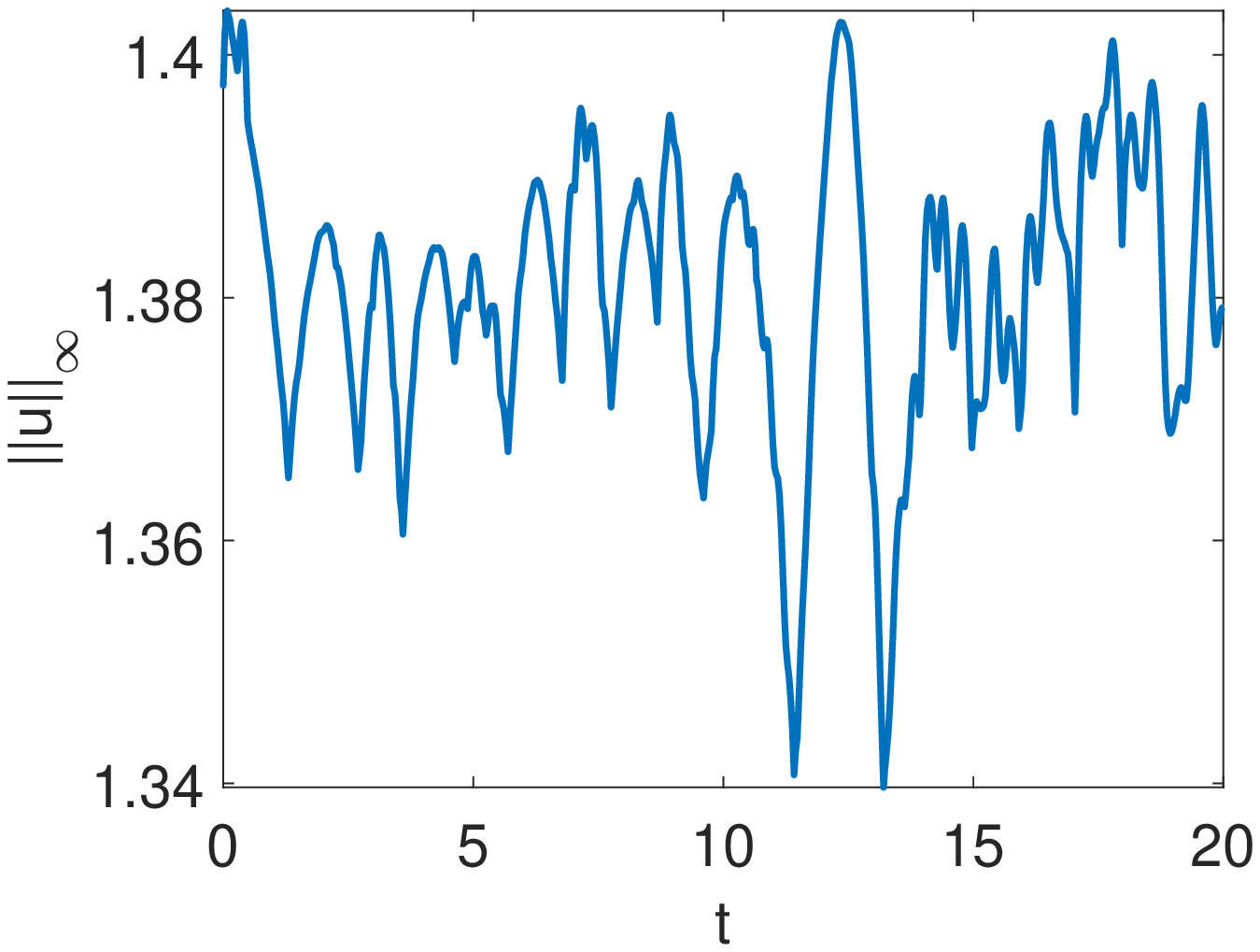}
	\end{minipage}
	\caption{The solution of the fKP-I equation that is perturbed with nonlocalized perturbation at $t=20$ for $\alpha=2$ and $c=0.4$ (left panel) and the change of the $L^\infty$-norm in time (right panel).}
	\label{P2alpha15stabil}
\end{figure}
Next, we investigate  the behaviour for the fKP-I equation with $\alpha=1.5$   numerically. We use $x_0=0$ in \eqref{pert2}. We first chose a $c$ value that is subcritical for KP-I equation. Figure \ref{P2alpha15instabil} shows the result when $c=2$.  The initial perturbation turns to growing peaks which move together with the wave. The numerical result indicates a  nonlinear instability.
The  experiments for $c=0.4$ is presented in Figure \ref{P2alpha15stabil}. Here the $L^\infty$-norm oscillates as in the KP-I equation with a subcritical $c$ value and we do not see the peaks which appear in the instable cases. The results may be interpreted as nonlinear stability.
From this experiment, we conjecture that there exists a critical speed $0.4<c^*_{1.5}<2$ for the fKP-I equation with $\alpha=1.5$. Notice that $c^*_{1.5}<c^*$, where $c^*$ is the critical speed for the classical KP-I equation.

\medskip

\subsection{Numerical experiments for the fKP-II equation}

The transverse stability of the classical KP-II equation has been shown in \cite{mizumachi1, mizumachi2}. To the best of our knowledge, there are no analytical results on the transverse (in)stability of line solitary waves for the fKP-II equation for general $\alpha$. It seems nontrivial to extend the methods used in \cite{mizumachi1, mizumachi2} to the fKP-II. For instance in \cite{mizumachi1} the author is able to find explicit modes of the linearized KP-II equation, which would not be possible for the fKP-II equation. Similarly, in \cite{mizumachi2} the authors relates the KP-II equation to a modified KP-II equation and proceed to show a global well-posedness result for that equation. In the fractional case the same procedure would lead to a modified fKP-II equation and showing global well-posedness for such an equation is likely to be more involved.

In this subsection, we present some numerical experiments to give an insight to the problem of transverse stability of fKP-II equation. Space intervals and number of grid points for the experiments are the same as in the fKP-I case. Here we show the   difference between the unperturbed solution and the perturbed solution at several times so the behaviour of the perturbation in time becomes more visible in a smaller scale.

\begin{figure}[H]
	\begin{minipage}[t]{0.45\linewidth}
		\includegraphics[width=3.1in,height=2.3in]{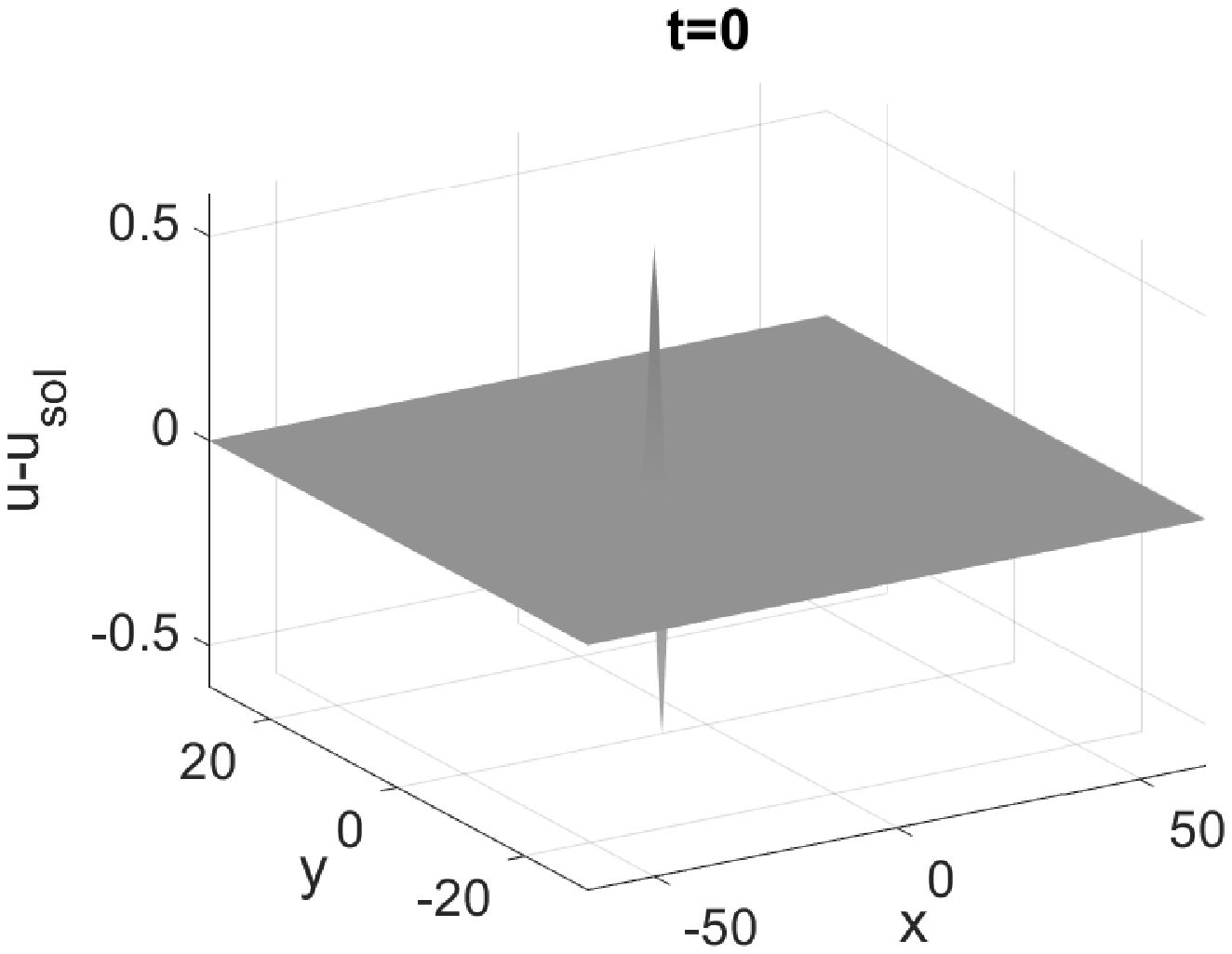}
	\end{minipage}
	\hspace{30pt}
	\begin{minipage}[t]{0.45\linewidth}
		\includegraphics[width=3.1in,height=2.3in]{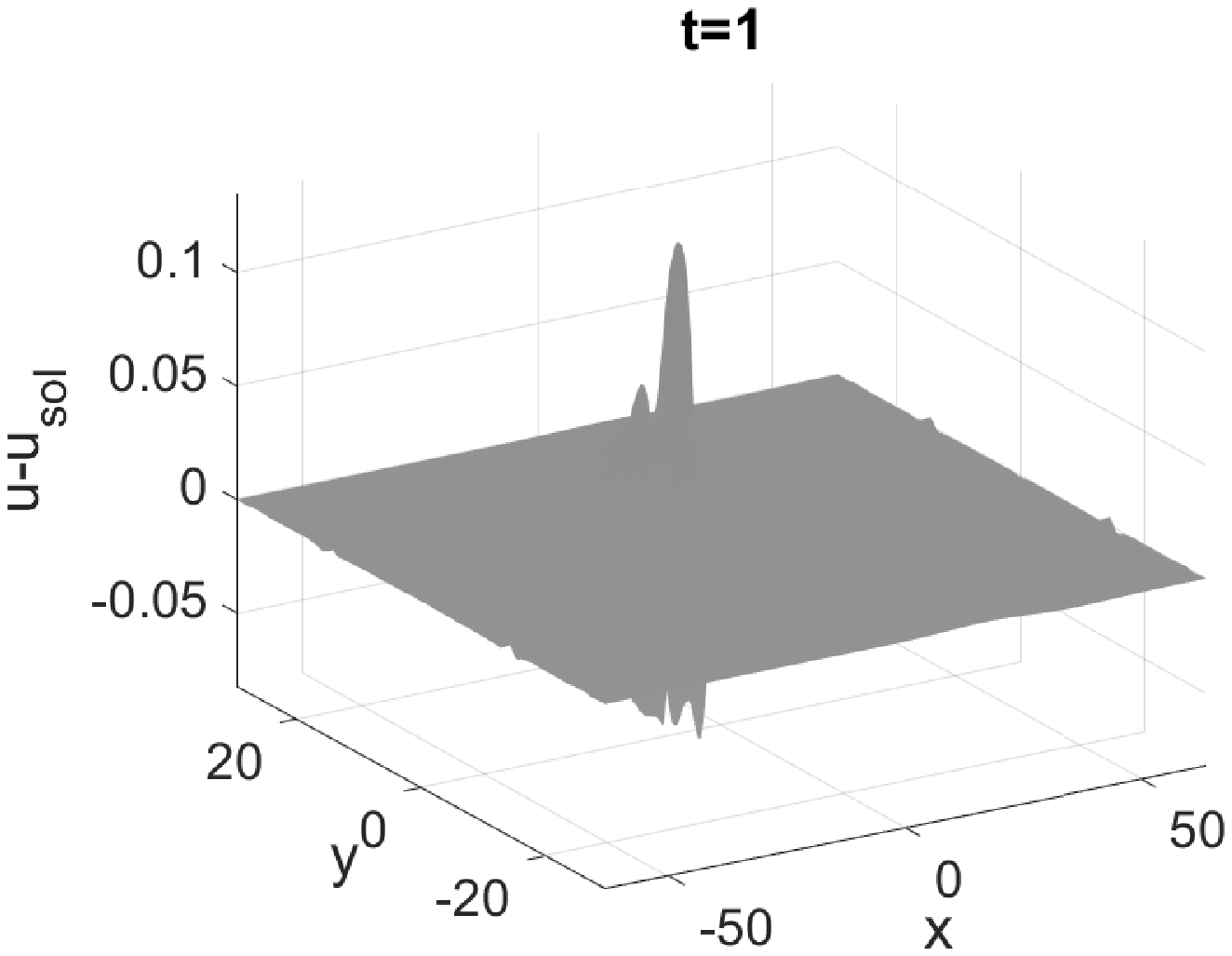}
	\end{minipage}
\end{figure}
\begin{figure}[H]
	\begin{minipage}[t]{0.45\linewidth}
		\includegraphics[width=3.1in,height=2.3in]{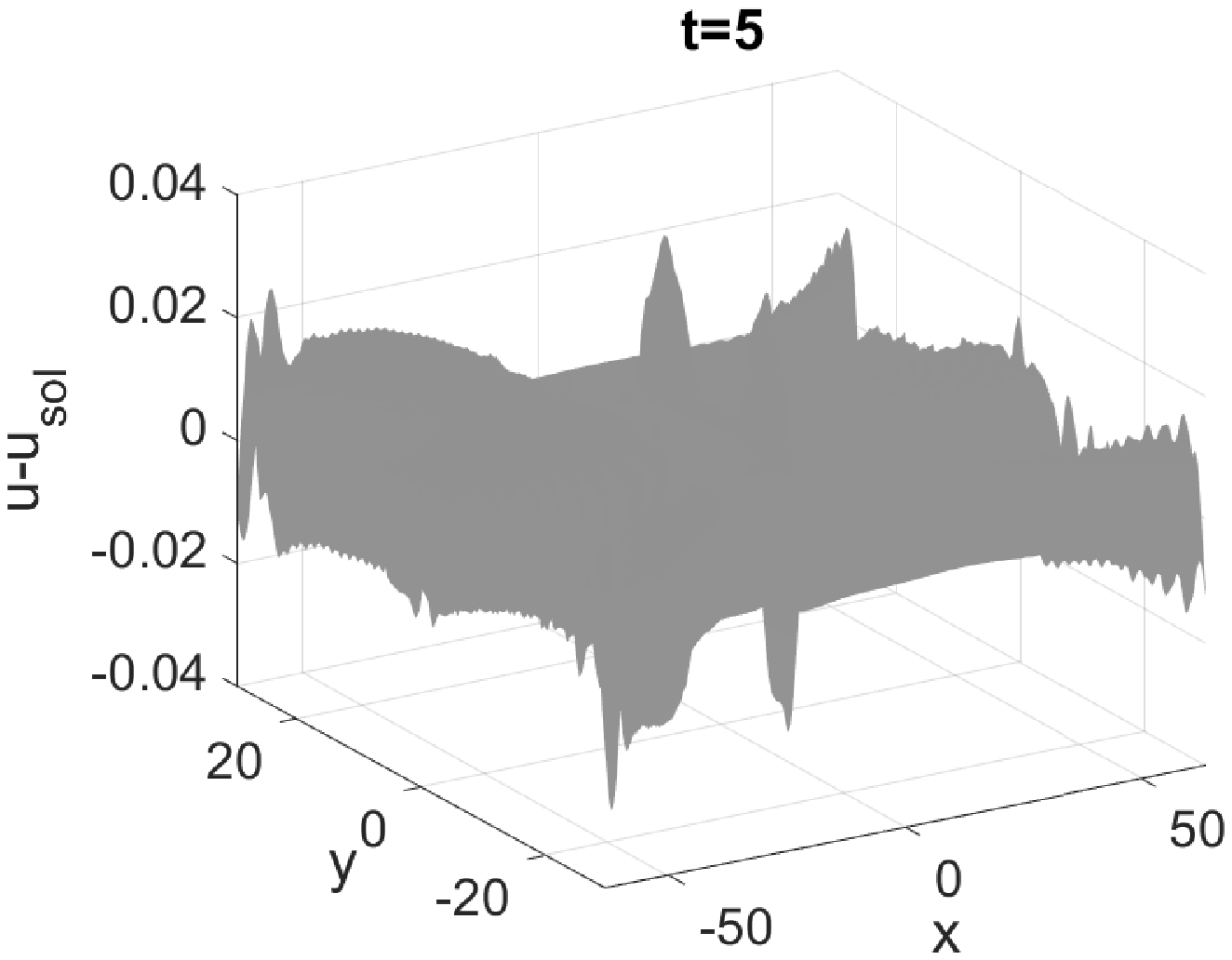}
	\end{minipage}
	\hspace{30pt}
	\begin{minipage}[t]{0.45\linewidth}
		\includegraphics[width=3.1in,height=2.3in]{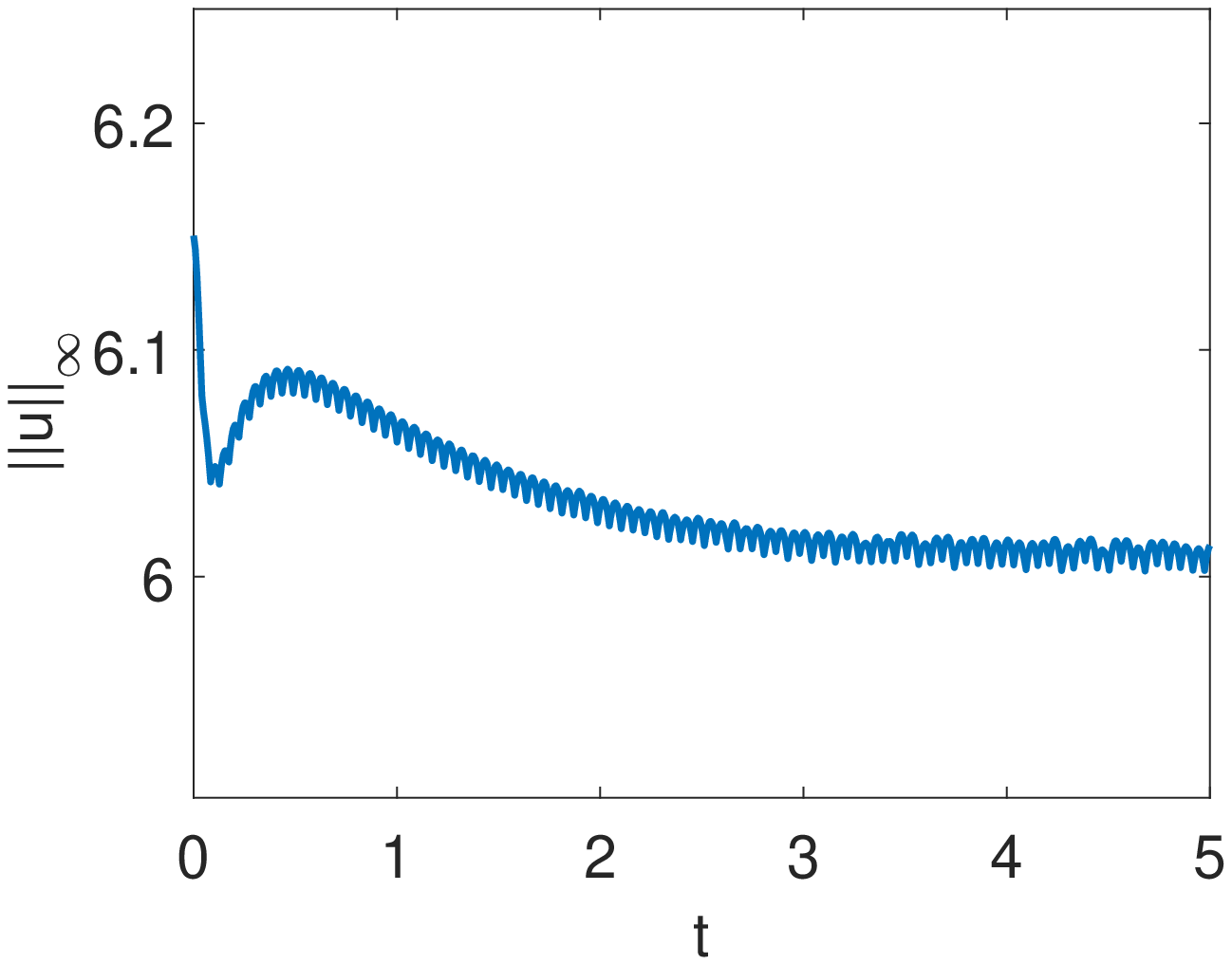}
	\end{minipage}
	\caption{The difference of perturbed solution and the solitary wave solution at several times and $L^\infty$ norm of the perturbed solution for the fKP-II equation when $\alpha=2$, $c=2$. }
	\label{KP2alpha2}
\end{figure}

We first consider the classical KP-II equation where $\alpha=2$ and present the results in Figure \ref{KP2alpha2}. We observe that the perturbation spreads out to the whole domain and gets smaller. The $L^\infty$-norm of the perturbed solution decreases and converges to the norm of the unperturbed solution. In Figure \ref{KP2alpha2} the graphs are given in different scales to show that perturbations still exist at later times but become very small. The numerical results are compatible with the nonlinear stability results  on the line solitary waves of KP-II equation in the literature.

Although there are no analytical results for the transverse stability properties of the fKP-II equation for general $\alpha$, similar behaviour with the case $\alpha=2$ is observed numerically. Figure \ref{kp2alpha135} shows the case where $\alpha=1.35$ that is just above the $L^2$-critical value $\alpha=4/3$. Next experiment for the fKP-II case is for an $L^2$-supercritical $\alpha$ value. Figure \ref{kp2alpha09} depicts the fKP-II equation with $\alpha=0.9$.  In both figures we do not show waves at $t=0$ as they are very similar with $\alpha=2$. The experiments for all $\alpha$ values show that the initial perturbation disperses to the whole domain and gets smaller by time. In Figures \ref{kp2alpha135} and \ref{kp2alpha09} we also  present the change in the $L^\infty $-norm of the wave  for $\alpha=1.35$ and $\alpha=0.9$ respectively. For both $\alpha$ values we observe that the effect of the perturbation vanishes by the time and the wave moves almost at a constant amplitude like a solitary wave. Therefore, the numerical experiments indicate a nonlinear stability of the solitary waves under the fKP-II flow.

   \begin{figure}[H]
  \begin{minipage}[t]{0.45\linewidth}
    \includegraphics[width=3.1in,height=2.3in]{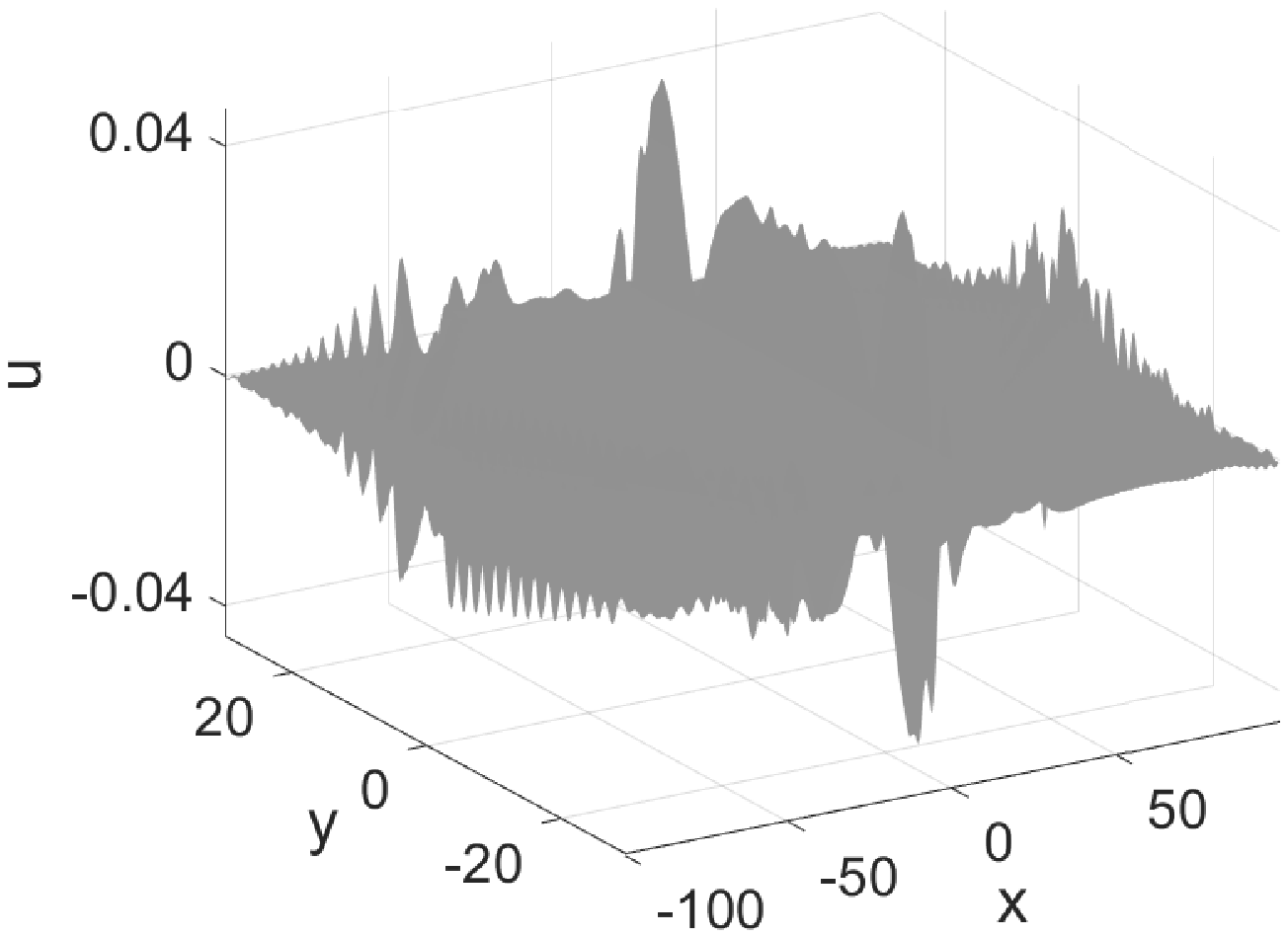}
  \end{minipage}
 \hspace{30pt}
 \begin{minipage}[t]{0.45\linewidth}
    \includegraphics[width=3.1in,height=2.3in]{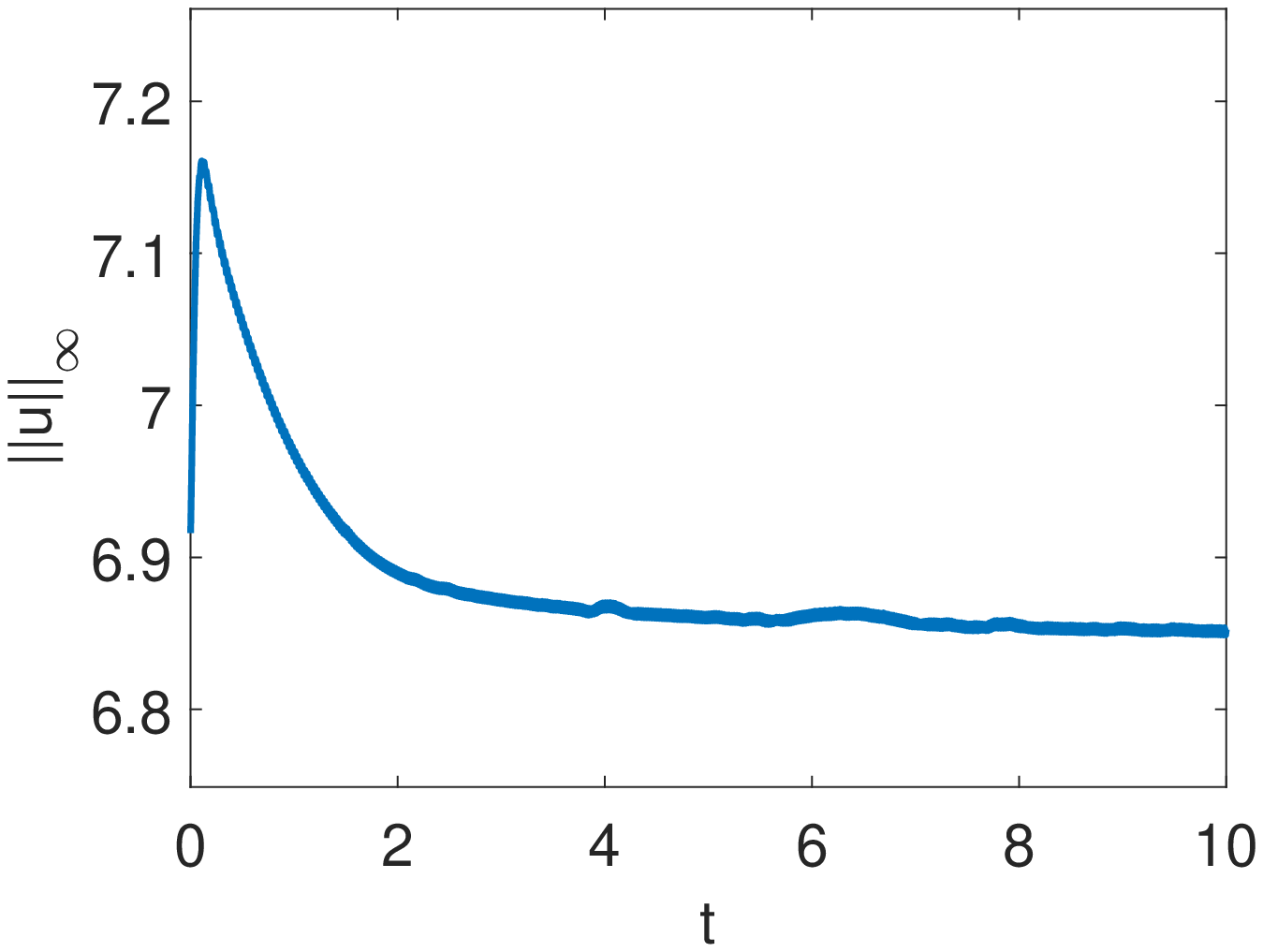}
  \end{minipage}
  \caption{The difference of perturbed solution and the solitary wave solution at $t=10$ (left panel)  and $L^\infty$ norm of the perturbed solution (right panel) for the fKP-II equation when $\alpha=1.35$, $c=2$.}
   \label{kp2alpha135}
 \end{figure}

    \begin{figure}[H]
  \begin{minipage}[t]{0.45\linewidth}
    \includegraphics[width=3.1in,height=2.3in]{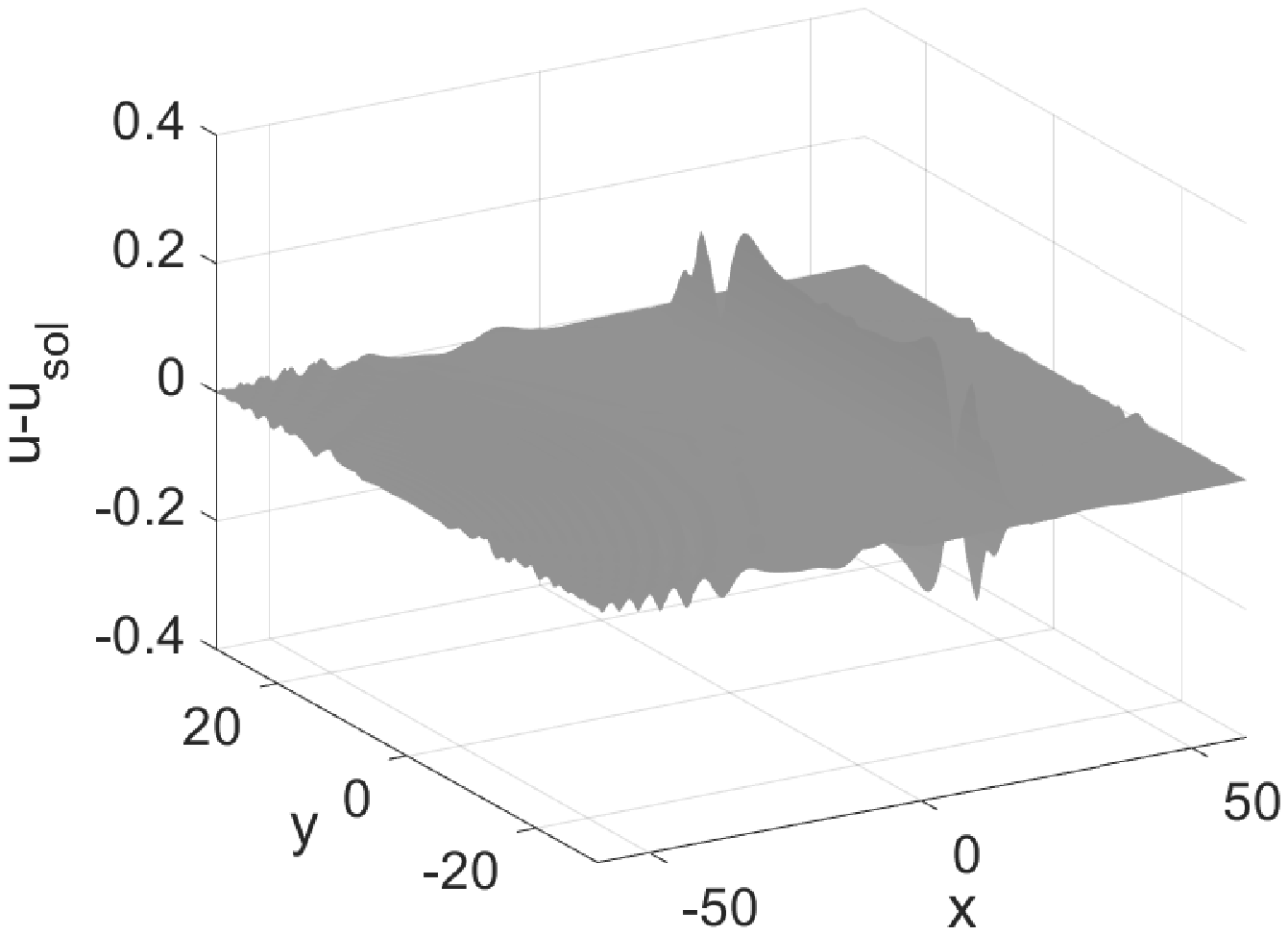}
  \end{minipage}
 \hspace{30pt}
 \begin{minipage}[t]{0.45\linewidth}
    \includegraphics[width=3.1in,height=2.3in]{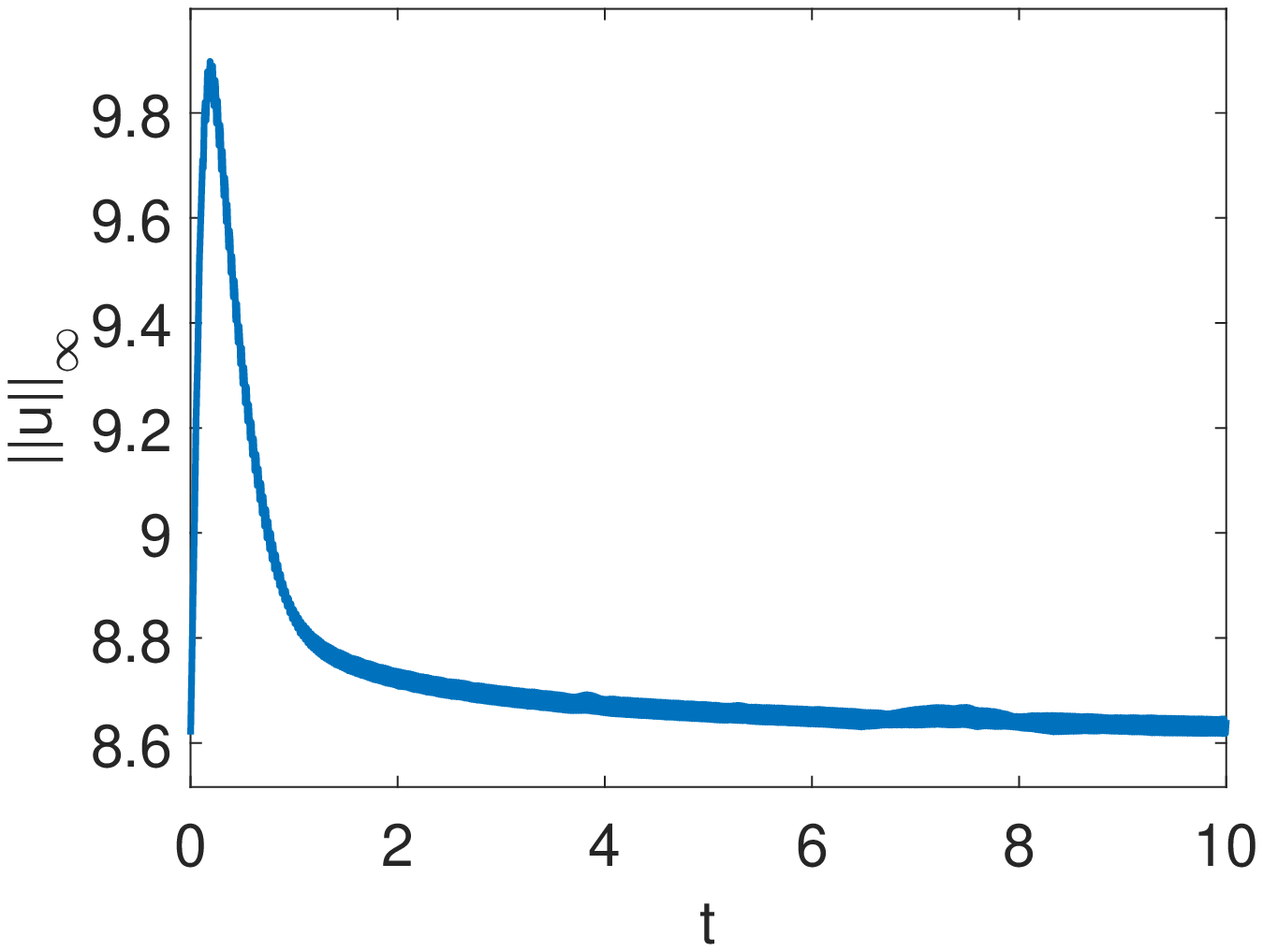}
  \end{minipage}
  \caption{The difference of perturbed solution and the solitary wave solution at $t=10$ (left panel)  and $L^\infty$ norm of the perturbed solution (right panel) for the fKP-II equation when $\alpha=0.9$, $c=2$.}
   \label{kp2alpha09}
 \end{figure}


Eventually, in Figure \ref{timeKP2} we present the time when the amplitude of the initial perturbation is halved. In the left panel we fix $c=2$ and consider several values of $\alpha$, then in right panel we fix $\alpha=1.5$ and consider various values of $c$. We see that the effect of perturbation vanishes faster for larger values of $\alpha$ and for smaller values of $c$.

    \begin{figure}[H]
  \begin{minipage}[t]{0.45\linewidth}
    \includegraphics[width=3.1in,height=2.3in]{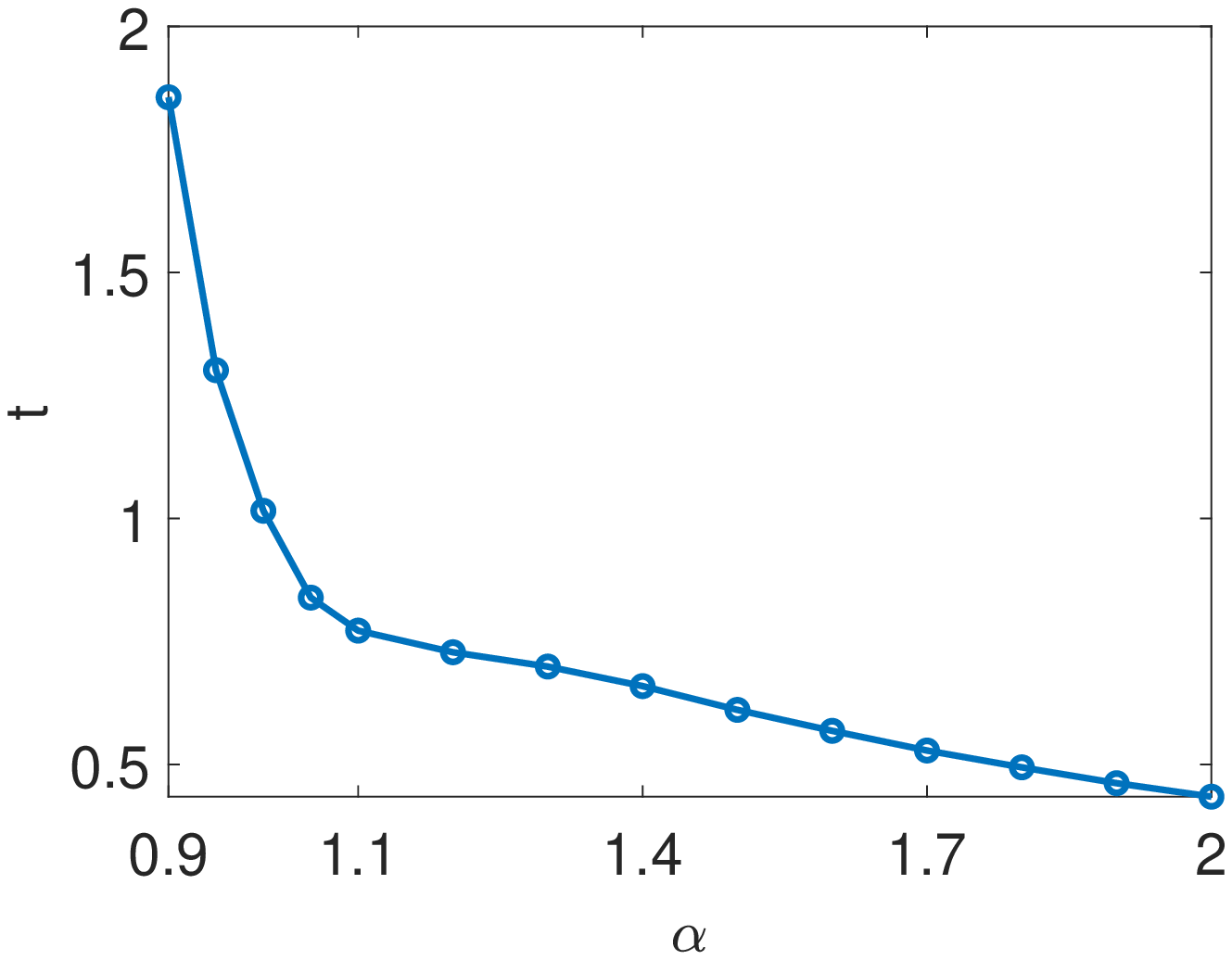}
  \end{minipage}
 \hspace{30pt}
 \begin{minipage}[t]{0.45\linewidth}
    \includegraphics[width=3.1in,height=2.3in]{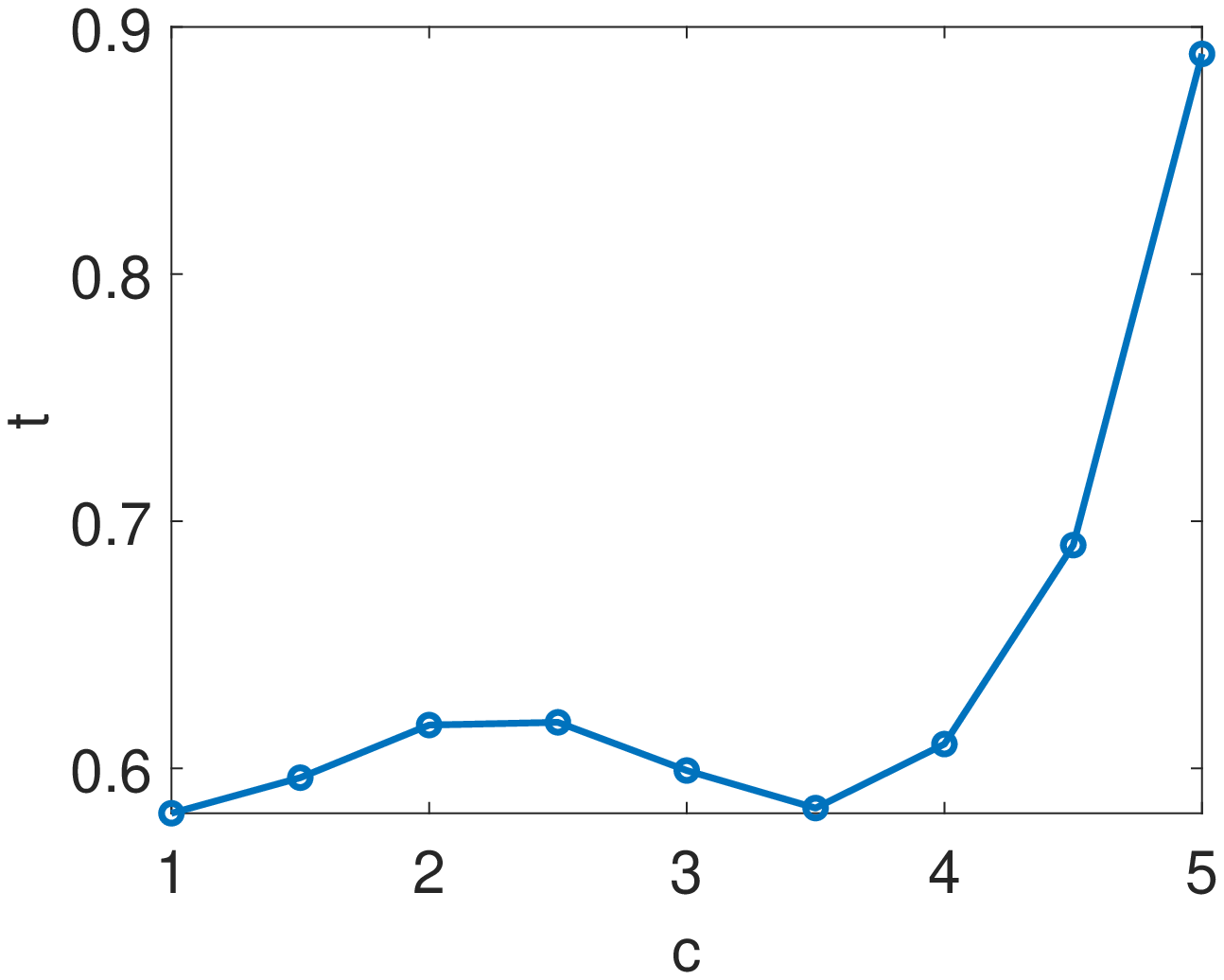}
  \end{minipage}
  \caption{The time that the amplitude of the initial perturbation is halved for various $\alpha$ values when $c=2$ (left panel) and for various values of $c$ when $\alpha=1.5$
(right panel).}
   \label{timeKP2}
 \end{figure}


\appendix

\section{Gagliardo--Nirenberg type inequality}

We formulate an auxiliary result in the spirit of Gagliardo--Nirenberg inequalities:

\begin{lemma}\label{lemmaA} If $0<a<b$ and $u\in H^b(\R)$, then for all $ n\in \N$, and for any $\e>0$ we have
	\[
	\|D^au\|^2= \e \sum_{m=0}^{n-1}\frac{1}{4^m}\|D^b u\|^2+ \frac{1}{4^n\e^{2^n-1}}\|D^{s_{n+1}}u\|^2,
	\]
	where $(s_n)_{n\in\N}$ is a recursively defined sequence with $s_1=a$, $s_{n+1}=2s_n-b$, and it is monotonously decreasing with $\lim_{n\to \infty}s_n=-\infty$.
	In particular, there exists $N\geq 1$ such that
	\[
	\|D^au\|^2\leq \e \sum_{m=0}^{N-1}\frac{1}{4^m}\|D^b u\|^2+ \frac{1}{4^n\e^{2^N-1}}\|u\|^2.
	\]
\end{lemma}

\begin{proof}
	The statement is proved by induction. If $n=1$, then
	\begin{align*}
		\|D^a u\|^2 &= \int_{\R} (|\xi|^a |\hat u(\xi)|)^2\, d\xi=\int_{\R}(|\xi|^b |\hat u(\xi)|)(|\xi|^{2a-b}|\hat u(\xi)|)\,d\xi = \e \|D^b u\|^2 + \frac{1}{4\e}\|D^{2a-b}u\|^2,
	\end{align*}
	for any $\e>0$,
	where we use Plancherel's identity, and Young's inequality. In view of $s_2=2s_1-b=2a-b$, the statement holds true for $n=1$.
	
	Now assume that for some $n\geq 1$
	\[
	\|D^au\|^2= \e \sum_{m=0}^{n-1}\frac{1}{4^m}\|D^b u\|^2+ \frac{1}{4^n\e^{2^n-1}}\|D^{s_{n+1}}u\|^2.
	\]
	Estimating
	\begin{align*}
		\|D^{s_{n+1}}\|^2&=\int_{\R} (|\xi|^{s_{n+1}} |\hat u(\xi)|)^2\, d\xi\\
		&=\int_{\R}(|\xi|^b |\hat u(\xi)|)(|\xi|^{2s_{n+1}-b}|\hat u(\xi)|)\,d\xi = \e_n \|D^b u\|^2 + \frac{1}{4\e_n}\|D^{2s_{n+1}-b}u\|^2,
	\end{align*}
	where $\e_n=\e^{2^n}$, we obtain that
	\begin{align*}
		\|D^au\|^2&= \e \sum_{m=0}^{n-1}\frac{1}{4^m}\|D^b u\|^2+ \frac{1}{4^n\e^{2^n-1}}\left( \e^{2^n} \|D^b u\|^2 + \frac{1}{4\e^{2^n}}\|D^{2s_{n+1}-b}u\|^2 \right)\\
		&=\e \sum_{m=0}^{n}\frac{1}{4^m}\|D^b u\|^2+\frac{1}{4^{n+1}\e^{2^{n+1}-1}}\|D^{s_{n+2}}u\|^2,
	\end{align*}
	which concludes the induction.
	As $a<b$ and $s_1=a$ we have $s_2=2s_1-b<s_1$. Repeating the argument, we obtain $s_{n+1}-s_n=s_n-b<0$, which therefore shows that $(s_n)_{n\in \N}$ forms a monotonously decreasing sequence.
	If the sequence is bounded from below, then $\lim_{n\to \infty} s_n = 2\lim_{n\to \infty}s_{n-1}-b$, but this would only be possible if $\lim_{n\to \infty}s_n=b$, which is a contradiction to the monotonic decay of the sequence. Hence $\lim_{n\to \infty}s_n=-\infty$.
	
	\medskip
	
	In order to prove the second assertion we use the fact that there exists $N\geq 1$ such that $s_{N}>0$, while $s_{N+1}<0$ as $s_1>0$ and $\lim_{n\to \infty} s_n=-\infty$. In this case, we estimate
	\begin{align*}
		\|D^{s_N}u\|^2= \int_{\R} (|\xi|^{s_N}\hat u(\xi))^2\, d\xi &= \int_{|\xi>1|}|\xi|^{b}\hat u(\xi) |\xi|^{2s_N-b}\hat u(\xi)\,d\xi + \int_{|\xi|\leq 1}(|\xi|^{s_N}\hat u(\xi))^2\, d\xi \\
		&\leq \int_{|\xi>1|}|\xi|^{b}\hat u(\xi) \hat u(\xi)\,d\xi + \int_{|\xi|\leq 1}(\hat u(\xi))^2\, d\xi,
	\end{align*}
	where $2s_N-b=s_{N+1}<0$. Then the results follows directly from Young's inequality.
\end{proof}

\bigskip

\noindent \textbf{Acknowledgements}\\
The authors would like to thank Christian Klein for helpful discussions and Mariana Haragus for her suggestions on the method. This research was carried out while G.B. was supported by the Deutsche Forschungsgemeinschaft (DFG, German Research Foundation) --
Project-ID 258734477 -- SFB 1173.

\end{document}